\documentclass[onefignum,onetabnum]{siamart190516}

\NewDocumentCommand{\dgal}{sO{}m}{%
  \IfBooleanTF{#1}
    {\dgalext{#3}}
    {\dgalx[#2]{#3}}%
}

\NewDocumentCommand{\dgalext}{m}{%
  \sbox0{%
    \mathsurround=0pt % just for safety
    $\left\{\vphantom{#1}\right.\kern-\nulldelimiterspace$%
  }%
  \sbox2{\{}%
  \ifdim\ht0=\ht2
    \{\kern-.45\wd2 \{#1\}\kern-.45\wd2 \}%
  \else
  \fi
}

\NewDocumentCommand{\dgalx}{om}{%
  \sbox0{\mathsurround=0pt$#1\{$}%
  \sbox2{\{}%
  \ifdim\ht0=\ht2
    \{\kern-.45\wd2 \{#2\}\kern-.45\wd2 \}%
  \else
    \mathopen{#1\{\kern-.5\wd0 #1\{}
    #2
    \mathclose{#1\}\kern-.5\wd0 #1\}}
  \fi
}

\usepackage{amsmath}
\usepackage{amsfonts}
\usepackage{amssymb}

\usepackage{graphicx}
\usepackage{algorithmic, algorithm}

\usepackage{cite}
\usepackage{bm}
\usepackage{subfigure}

\ifpdf
  \DeclareGraphicsExtensions{.eps,.pdf,.png,.jpg}
\else
  \DeclareGraphicsExtensions{.eps}
\fi

\newsiamremark{remark}{Remark}
\newsiamremark{hypothesis}{Hypothesis}
\crefname{hypothesis}{Hypothesis}{Hypotheses}
\newsiamthm{claim}{Claim}

\headers{PPR for PHT-splines}{Ying Cai, Falai Chen, Hailong  Guo, Hongmei Kang and Zhimin Zhang}

\title{Polynomial preserving recovery for PHT-splines}

\author{Ying Cai\thanks{School of Mathematical Sciences, University of Science and Technology of China, Hefei 230026, China (\email{ycai@ustc.edu.cn}).}
\and Falai Chen\thanks{School of Mathematical Sciences, University of Science and Technology of China, Hefei 230026, China (\email{chenfl@ustc.edu.cn}).}
\and Hailong Guo\thanks{School of Mathematics and Statistics, The University of Melbourne, Parkville, VIC 3010, Australia (\email{hailong.guo@unimelb.edu.au}).}
\and Hongmei Kang\thanks{School of Mathematical Sciences, Soochow University, No. 1 Street Shizi, Suzhou, 215006, China (\email{khm@suda.edu.cn}).}
\and Zhimin Zhang\thanks{Department of Mathematics, Wayne State University, Detroit, MI 48202, USA (\email{ag7761@wayne.edu}).}
}

\usepackage{amsopn}

\ifpdf
\hypersetup{
  pdftitle={Superconvergence for pht-spline},
  pdfauthor={Ying Cai}
}
\fi

\newcommand{\comp}{\mathrm{comp}}
\newcommand{\ssubset}{\subset\subset}
\begin{document}

\maketitle

% REQUIRED
\begin{abstract}
We propose a polynomial preserving recovery method for PHT-splines within isogeometric analysis to obtain more accurate gradient approximations. The method fully exploits the local interpolation properties of PHT-splines and avoids the need for information on gradient superconvergent points. By leveraging the superconvergence argument of difference quotients and the interior error estimate, we establish the superconvergence property of the recovered gradient on translation invariant meshes. As a byproduct, a recovery-based \textit{a posteriori} error estimator is developed for adaptive refinement. Numerical results confirm the theoretical findings and demonstrate the effectiveness of the proposed method.
\end{abstract}

% REQUIRED

% REQUIRED

\section{Introduction}
\label{sec:int}
Isogeometric analysis (IGA),  introduced by Hughes et al.\ in their seminal work~\cite{Hughes2005}, establishes a unified framework to integrate Computer-Aided Design (CAD) with Computer-Aided Engineering (CAE). The fundamental principle of IGA lies in the consistent utilization of spline basis functions throughout both the geometric modeling and numerical analysis phases. This paradigm eliminates the conventional conversion process from CAD’s spline representations to piecewise linear finite element meshes, thereby enhancing computational efficiency and geometric accuracy in engineering simulations. Readers interested in the mathematical foundations of IGA may consult~\cite{IGAbook2009} for further details.

To accurately represent complex geometries, tensor-product splines often require a large number of redundant control points to achieve local mesh refinement. To overcome this limitation, the so-called T-splines were introduced \cite{Sederberg2003,Sederberg2004}, which allow T-junctions in the control mesh and provide greater flexibility in geometric modeling. However, as pointed out by Buffa et al. \cite{BCS2010}, the blending functions of T-splines may suffer from linear dependence. To address this issue, some researchers proposed analysis-suitable T-splines (AST-splines) \cite{SLSH2012,LS2014}, which form a subset of T-splines with desirable properties such as partition of unity, local support, and linear independence—features that are essential for both design and analysis. Nevertheless, the local refinement of AST-splines may extend beyond the region of interest, potentially increasing the number of control points. For a more comprehensive overview of locally refined splines in isogeometric analysis, we refer the reader to the review article \cite{survey}.

Among the various locally refinable spline technologies, polynomial splines over hierarchical T-meshes (PHT-splines) have gained significant attention due to their perfect local refinement property. PHT-splines were first introduced by Deng et al. in \cite{Deng2008}, where the dimension of the spline space and the explicit construction of the basis functions were provided. The newly constructed basis functions possess several desirable properties similar to those of B-splines, including non-negativity, local support, and partition of unity. Owing to these favorable mathematical characteristics, PHT-splines have found diverse applications in areas such as surface reconstruction \cite{wang2010adaptive} and the discretization of partial differential equations \cite{Tian2011,NN2011,NK2011,BGAR2022} and so on, among others.

In many mathematical models arising in scientific computing, the gradient of the solution plays a critical physical role. In fact, in various applications, the gradient is often of greater significance than the solution itself. Consequently, achieving high-accuracy gradient approximations is of great importance in both scientific and engineering contexts. In classical finite element methods, gradient recovery is widely recognized as one of the most important post-processing techniques, aiming to reconstruct a more accurate approximation of the gradient than the one directly obtained from the finite element solution. In the finite element framework, gradient recovery has been extensively investigated (\cite{ZN2005,NZ2004,NZ2005,ZZ1992,ZZ21992,GUO2025}), and various types of recovery techniques have been proposed and analyzed. Representative approaches include simple or weighted averaging methods (\cite{ZZ1992}), the superconvergent patch recovery (SPR) technique (\cite{ZZ1992,ZZ21992}), and the polynomial preserving recovery (PPR) method
(\cite{ZN2005,NZ2004,NZ2005,GZZ2017}). These gradient recovery techniques, particularly SPR and PPR, have been widely adopted in commercial finite element software—for example, SPR in Abaqus and LS-DYNA, and PPR in COMSOL Multiphysics.

Although gradient recovery techniques have demonstrated substantial effectiveness within the finite element method, analogous post-processing strategies have received comparatively limited attention in the context of isogeometric analysis (IGA). To the best of our knowledge, the only existing contributions in this direction are due to Mukesh et al.~\cite{KKJ2017,AKJKK2024}, who developed gradient recovery methods—primarily based on the superconvergent patch recovery (SPR) technique—for locally refined B-splines (LR B-splines). The principal aim of the present work is to develop a polynomial preserving recovery (PPR) method tailored to PHT-splines and to rigorously establish its superconvergence properties.

Our proposed method has several notable advantages. First, when constructing local patches, the selection of sampling points theoretically only requires the mesh basis vertices, while in practical computations, more flexible choices can be made.
 In contrast, the SPR technique typically requires sampling points to coincide with the gradient superconvergent points of the numerical solution, whose locations are often difficult to determine. In the context of isogeometric analysis, research on superconvergence is still in its early stages, and the theoretical foundation remains incomplete, especially for complex mesh configurations. In \cite{KKJ2017}, the existence and computation of the gradient superconvergent points were studied based on finite element analogies, and the proof relies on computer-based verification (see \cite{BS2001}). This highlights a key advantage of the PPR method. Second, this work focuses on PHT-splines, which offer a favorable local structure and admit an explicitly constructed interpolation operator. This operator requires only geometric information at the basis vertices to perform interpolation, greatly simplifying the process of assembling the recovered gradient into a global representation. Specifically, it suffices to solve a small local linear system of size $4\times4$, as opposed to techniques like conjoining the polynomial expansions in \cite{BB1994} used for LR B-splines.

In addition, another key contribution of this work lies in the superconvergence analysis of the recovered gradient. By employing the argument of difference quotient \cite{NS1974, Wa1995}, together with interior error estimates \cite{NS1974,SW1977,SW1995}, we establish a superconvergence result for the recovered gradient on translation-invariant meshes. Finally, as a byproduct of our analysis, we present a recovery type \textit{a posteriori} error estimator to support an adaptive refinement strategy based on PHT-splines.

The rest of the paper is organized as follows. In Section \ref{sec:notations}, we briefly review the construction of PHT-spline basis functions and the associated interpolation operator. Section \ref{sec:recovery} introduces the polynomial preserving recovery (PPR) operator, analyzes its main properties, and further discusses its extension to general physical domains as well as the development of a recovery-based \textit{a posteriori} error estimator. Section \ref{sec:super} presents a superconvergence analysis of the recovered gradient on translation-invariant meshes. Several numerical examples are provided in Section \ref{sec:num} to validate the theoretical results. Finally, conclusions are drawn in the last section.

\section{Notations and PHT-splines}\label{sec:notations}
In this section, we introduce some fundamental notations and then review the construction of PHT-splines.

Let $\Omega \subset \mathbb R^2$ be a rectangular domain with boundary $\partial\Omega$. Throughout this paper, we employ the standard notation for Sobolev spaces and their associated (semi)norms as presented in \cite{Ci1978}. Given a subdomain $\mathcal D \subset \Omega$, we denote by $\mathbb P_m(\mathcal D)$ the space of polynomials of degree less than or equal to $m$ defined on $\mathcal D$. The Sobolev space $W_p^k(\mathcal D)$ is equipped with the norm $\|\cdot\|_{k,p,\mathcal D}$ and seminorm $|\cdot|_{k,p,\mathcal D}$. In the particular case $p = 2$, we adopt the standard notation $H^k(\mathcal D) = W_2^k(\mathcal D)$, omitting the subscript $p$. $(\cdot,\cdot)_{\mathcal D}$ denotes the standard $L^2$ inner product on $\mathcal D$.

\subsection{Hierarchical T-mesh}
A T-mesh over the domain $\Omega$ is a rectangular partition formed by axis-aligned mesh lines matching the domain boundaries, while admitting T-junctions (see \cite{Sederberg2003, Deng2008}). Every element in this partition must be strictly rectangular. A characteristic T-mesh configuration is depicted in Fig. \ref{fig:Tmesh}.
\begin{figure}[ht]
\centering
\includegraphics[width=0.36\textwidth]{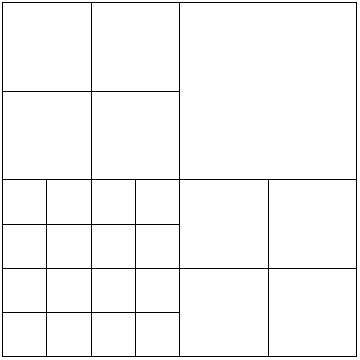}
\caption{An example of the T-mesh.}
\label{fig:Tmesh}
\end{figure}

A hierarchical T-mesh is a specialized T-mesh structure with an inherent level-based hierarchy. The construction begins with an initial tensor-product mesh $\mathcal T_0$ (level 0). At each refinement stage $k\geq0$, the T-mesh $\mathcal T_k$ is refined by uniformly subdividing selected elements into four congruent subelements, which then form the subsequent level mesh $\mathcal T_{k+1}$. This dyadic refinement process generates a nested sequence of meshes, where the mesh size at level $k$ satisfies $h_k = h_0 \cdot 2^{-k}$, with $h_0$ being the initial mesh size. Fig.~\ref{fig:hierarchicalT-mesh} demonstrates a typical hierarchical T-mesh configuration.
\begin{figure}[ht]
    \centering
    \subfigure{
        \includegraphics[width=0.3\textwidth]{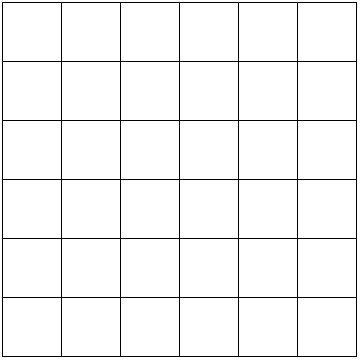}
    }
    \subfigure{
	\includegraphics[width=0.3\textwidth]{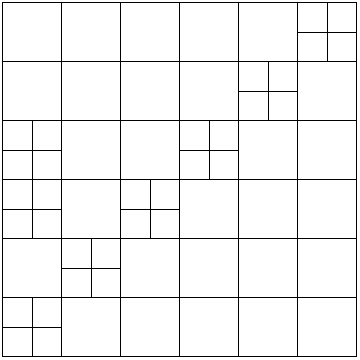}
    }
        \subfigure{
	\includegraphics[width=0.3\textwidth]{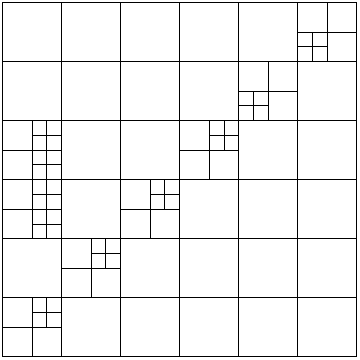}
    }
\caption{Example for hierarchical T-mesh (level 0, level 1 and level 2).}
    \label{fig:hierarchicalT-mesh}
\end{figure}

Consider a bivariate polynomial spline space defined over a two-dimensional T-mesh $\mathcal T$. The PHT-spline space of bi-degree $(m,n)$ with continuity orders $\alpha$ and $\beta$ along the $x$- and $y$-directions, respectively, is defined as
\[
\mathcal S(m,n,\alpha,\beta,\mathcal T) :=\{ s \in C^{\alpha,\beta}(\Omega): s|_K \in \mathbb P_{mn}(K), \, \forall K \in \mathcal T\},
\]
where $\mathbb P_{mn}$ denotes the space of bivariate polynomials of degree $m$ in $x$ and $n$ in $y$.

The spline space exhibits favorable theoretical and practical properties when the polynomial degrees satisfy the inequalities $m \geq 2\alpha+1$ and $n \geq 2\beta+1$, under which the dimension of the space can be determined explicitly \cite{DCF2006}. A canonical example is the $C^1$-continuous bi-cubic PHT-spline space $\mathcal S(3,3,1,1,\mathcal T)$, whose dimension is given by the compact formula
\begin{equation}\label{eq:dimformula}
\dim \mathcal{S}(3,3,1,1,\mathcal T) = 4(V^b + V^+),
\end{equation}
where $V^b$ denotes the number of boundary vertices and $V^+$ denotes the number of interior crossing vertices of $\mathcal T$.
These vertices are collectively referred to as basis vertices, with each contributing exactly four basis functions to the spline space, as reflected in the dimension formula \eqref{eq:dimformula}. In this work, we focus on cubic PHT-splines defined over the mesh $\mathcal T_h$, and denote the associated spline space concisely by $V_h$, where $h$ represents the characteristic mesh size. The set of all basis vertices is denoted by $\mathcal N_h$, and we let $|N_h|$ represent the total number of basis vertices in $\mathcal T_h$.

For each basis vertex  $z_i \in \mathcal N_h$, the corresponding four basis functions are denoted by ${\phi_{z_i}^1, \phi_{z_i}^2, \phi_{z_i}^3, \phi_{z_i}^4}$. For notational simplicity, we define the vector-valued function $\bm\phi_{z_i} = (\phi_{z_i}^1\,\, \phi_{z_i}^2\,\, \phi_{z_i}^3\,\, \phi_{z_i}^4)^T$.

\subsection{The basis function construction} With the dimension of space $V_h$ established, the explicit construction of basis functions becomes crucial. It should be noted that the construction of basis functions for PHT-spline spaces is not uniquely determined, as demonstrated in several existing works \cite{Deng2008,KXCD2015,ZC2017}.

Here, we mainly adopt the construction method of the original basis functions for PHT-splines as described in \cite{Deng2008}. The methodology builds basis functions through an level-by-level process. Assume that the hierarchical T-mesh $\mathcal T_h$ is obtained by hierarchically refining the initial tensor product mesh $\mathcal T_0$ for $J$ times. Starting with the initial tensor product mesh $\mathcal T_0$, the construction employs standard bicubic $C^1$ tensor product B-splines as the basis functions.

Consider a hierarchical T-mesh at level $k$, $0\leq k\leq J-1$, denote as $\mathcal T_k$, where the basis functions $\{\phi_j^k\}$, $j=1,2,\cdots,d_k$, have already been established. The construction of basis functions at level $\mathcal T_{k+1}$ consists of two key operations:
(i) truncation of the existing basis functions $\{\phi_j^k\}$ defined on $\mathcal T_k$, followed by
(ii) generation of new bicubic $C^1$-continuous B-spline basis functions corresponding to the new basis vertices in $\mathcal T_{k+1}$.

During the truncation step, let $\theta_i^k$ ($i = 1, 2,\cdots,C_k$) denote all level $k$ elements undergoing subdivision. For every basis function
$\phi_j^k$ ($j\in\{1,2,\cdots,d_k\}$) that maintains non-zero values within any $\theta_i^k$ element, we express it using Bézier representations across these elements at the refined level $k+1$. Importantly, while $\phi_j^k$ retains its original functional form, its definition now extends over the updated mesh $\mathcal T_{k+1}$.

The mesh refinement from $\mathcal T_k$ to $\mathcal T_{k+1}$ introduces additional basis vertices. Within each level $k + 1$ element, the 16 Bézier ordinates are systematically partitioned into four groups, each corresponding to a specific element vertex. This enables the modification of all level $k$ basis functions $\{\phi_j^k\}_{j=1}^{d_k}$ to $\{\widetilde \phi_j^k\}_{j=1}^{d_k}$ through the following procedure: for every basis function, all Bézier ordinates associated with newly introduced vertices are set to zero. It is easy to check that $\{\widetilde \phi_k^j\}_{j=1}^{d_k}$ contain in the spline space $\mathcal S(3,3,1,1,\mathcal T_{k+1})$.

The second construction step generates four standard tensor-product B-spline basis functions for each newly created vertex at level $k+1$. The resulting basis system maintains several crucial properties: non-negativity, linear independence, partition of unity, and compact local support.

\subsection{Interpolation operator for PHT-splines}
To facilitate the construction of the PPR operator in the following section, we first recall the definition of the interpolation operator $I_h$ associated with the spline space $V_h$. This operator relies on the geometric information of the target function evaluated at the basis vertices.

To this end, for any sufficiently smooth function $f(\bm x)$, we introduce a geometric information operator $\mathcal G$ defined by
\[\mathcal Gf(\bm x)=\left(f(\bm x),\,\, \frac{\partial f(\bm x)}{\partial x},\,\,
\frac{\partial f(\bm x)}{\partial y},\,\,\frac{\partial^2f(\bm x)}{\partial x\partial y}\right)^T.\]
By construction, $\mathcal G$ is a linear operator, and it satisfies the following key property: for any basis function $\phi$ in $V_h$, associated with a basis vertex $z_0$, the evaluation $\mathcal G\phi(z_i)$ vanishes at all other basis vertices $z_i \ne z_0$.

With the aid of the operator $\mathcal G$, we now define the interpolation operator $I_h$ for PHT-splines. Given a sufficiently smooth function $f$, we assume the interpolant takes the form
\begin{equation}\label{eq:ip}
I_hf(\bm x)=\sum_{j=1}^{|N_h|}\bm c_j^T\bm\phi_j(\bm x),
\end{equation}
where the coefficient vectors $\bm c_j \in \mathbb R^4$ are to be determined. For each fixed basis vertex $z_j$, associated with the basis function vector $\bm\phi_j$, we apply the operator $\mathcal G$ to both sides of equation \eqref{eq:ip} and evaluate at $z_j$, yielding
\[\mathcal Gf(z_j)=\mathbf B_j \bm c_j,\]
where $\mathbf B_j=\mathcal G\bm\phi_j^T(z_j)$ is a $4\times4$ matrix.
Suppose the control points for the four basis functions associated with a basis vertex are chosen in the following order:

\begin{tabular}{l}
$(x_{i-1}, x_{i-1}, x_i, x_i, x_{i+1})\times(y_{j-1}, y_{j-1}, y_j, y_j, y_{j+1})$, \\
$(x_{i-1}, x_i, x_i, x_{i+1}, x_{i+1})\times(y_{j-1}, y_{j-1}, y_j, y_j, y_{j+1})$, \\
$(x_{i-1}, x_i, x_i, x_{i+1}, x_{i+1})\times(y_{j-1}, y_j, y_j, y_{j+1}, y_{j+1})$, \\
$(x_{i-1}, x_{i-1}, x_i, x_i, x_{i+1})\times(y_{j-1}, y_j, y_j, y_{j+1}, y_{j+1})$.
\end{tabular}

Then, following the formulation in \cite{Tian2011}, the matrix $\mathbf B_j$ takes the form:
\[\mathbf B_j =
\begin{pmatrix}
(1 - \lambda)(1 - \mu) & \lambda(1 - \mu) & \lambda\mu & (1 - \lambda)\mu \\
-\alpha(1 - \mu) & \alpha(1 - \mu) & \alpha\mu & -\alpha\mu \\
-\beta(1 - \lambda) & -\beta\lambda & \beta\lambda & \beta(1 - \lambda) \\
\alpha\beta & -\alpha\beta & \alpha\beta & -\alpha\beta
\end{pmatrix},
\]
with
\[\alpha = \frac{1}{\Delta u_1 + \Delta u_2}, \quad
\beta = \frac{1}{\Delta v_1 + \Delta v_2}, \quad
\lambda = \alpha\Delta u_1, \quad
\mu = \beta\Delta v_1,\]
and the four neighbor elements around $z_0$ are with the sizes
$3\Delta u_1\times 3\Delta v_1,3\Delta u_2\times 3\Delta v_1,3\Delta u_1\times 3\Delta v_2,3\Delta u_2\times 3\Delta v_2$.
It follows that $\mathbf B_j$ is invertible, allowing the interpolation coefficients to be computed as $\bm c_j = \mathbf B_j^{-1} \mathcal Gf(z_j)$.

\section{Superconvergent  Gradient Recovery Method}\label{sec:recovery}
In this section, we present a high-accuracy and computationally efficient post-processing strategy tailored for the PHT-spline framework. Specifically, we introduce the polynomial preserving recovery (PPR) operator, denoted by $G_h$, which acts as a linear mapping from $V_h$ to $V_h \times V_h$.
Given a discrete function $u_h \in V_h$, a key observation is that the recovered gradient $G_h u_h$ also lies within the PHT-spline space. As a result, the construction of $G_hu_h$ reduces to the evaluation of certain geometric quantities at the basis vertices. This localization feature significantly simplifies the implementation and maintains the structure of the underlying spline space.

\subsection{Description of the algorithm}
Inspired by the polynomial preserving recovery (PPR) algorithms developed for standard finite elements in \cite{NZ2004,ZN2005,NZ2005}, our PPR scheme for PHT-splines is similarly organized into three main steps: (i) construction of local element patches; (ii) execution of local recovery procedures; and (iii) assembly of the recovered values into a global representation.

The construction of local element patches begins with the definition of basis vertex adjacent mesh layers. Let $z_i \in \mathcal N_h$ be a basis vertex. For any positive integer $n$, we define the $n$-th layer $\mathcal L(z_i,n)$ recursively as
\[\mathcal L(z_i,n):=
\begin{cases}
  z_i, & \mbox{if } n=0  \\
  \bigcup\limits_{\substack{T\in\mathcal T_h; \\ T\cap\mathcal L(z_i,n-1)\neq\emptyset}}
  T, & \mbox{if } n\geq 1.
\end{cases}\]
The set $\mathcal L(z_i,n)$ represents the union of elements forming an $n$-ring patch around the vertex $z_i$, i.e., all elements that are within $n$ adjacency layers of $z_i$.

For any basis vertex $z_i$, we construct the set $B_{z_i}$ of sampling points in its neighborhood. In contrast to the SPR method developed in previous works, which relies on the existence of gradient superconvergence points of the spline function, our approach avoids such limitations. A natural theoretical choice for defining $B_{z_i}^n$ is to consider the set of all basis vertices contained in $\mathcal L(z_i,n)$.

Let $n_i$ be the smallest integer such that the basis vertices in $B_{z_i}^n$ suffice to uniquely determine a quartic polynomial in the least-squares sense. We then define $B_{z_i} := B_{z_i}^{n_i}$ and the associated local patch $\Omega_{z_i} := \mathcal L(z_i,n_i)$. To construct the recovered gradient at a given basis vertex $z_i$, let $I_i$ denote the index set corresponding to the sampling points in $B_{z_i}$. A quartic polynomial $p_{z_i}$ centered at $z_i$ is then computed via the following least-squares fitting procedure:
 \[p_{z_i} = \mathop{\text{arg min}}\limits_{p\in\mathbb P_4(\Omega_{z_i})}\sum_{j\in I_i}|p(z_{i_j}) - u_h(z_{i_j})|^2,\]
where $z_{i_j} \in B_{z_i}$ are the sampling points.

\begin{remark}\label{rmk:sp}
We emphasize that the above construction of the sampling point set $B_{z_i}$ is primarily designed for theoretical analysis. In practical computations, ensuring a sufficient number of sampling points under this strategy typically requires a relatively large value of $n_i$, which may be computationally inefficient. To address this, one may enrich the set $B_{z_i}$ by including additional sampling points within each element of $\mathcal L(z_i,n)$, such as element centers, edge midpoints, or Gauss integration points.
\end{remark}

To avoid numerical instability in computations, let
\[h_i=\max\{|z_{i_m}-z_{i_\ell}|: i_m, i_\ell\in I_i\},\]
and define the local coordinate transform
\[F:(x,y)\rightarrow(\xi,\eta)=\frac{(x,y)-(x_i,y_i)}{h_i},\]
where $z=(x,y)$ and $z_i=(x_i,y_i)$ and $\hat z=(\xi,\eta)$. All the computations are performed at the reference
local element patch $\hat\Omega_{z_i}=F(\Omega_{z_i})$. Then, we can rewrite $p_{z_i}$ as
\[p_{z_i}=\mathbf p^T\mathbf a=\hat{\mathbf p}^T\hat{\mathbf a},\]
where
\begin{align*}
&\mathbf p^T=(1, x, y, x^2, xy, y^2, \cdots, xy^3, y^4) &\mathbf a=(a_1, a_2, a_3, a_4, \cdots,  a_{15})^T\\
&\hat{\mathbf p}^T=(1, \xi, \eta, \xi^2, \xi\eta, \eta^2, \cdots, \xi\eta^3, \eta^4)
&\hat{\mathbf a}=(a_1, h_ia_2, h_ia_3, h_i^2a_4, \cdots,  h_i^4a_{15})^T.
\end{align*}
Let $\hat z_{i_j}=F(z_{i_j})$. The coefficient vector $\hat{\mathbf a}$ is determined by solving the linear system
\[(\hat {\mathbf A}^T\hat {\mathbf A})\hat{\mathbf a}=\hat {\mathbf A}^T\hat{\mathbf b},\]
where
\begin{equation*}
	\hat{\mathbf A}
	= \begin{pmatrix}
		1 & \xi_{i_0} & \eta_{i_0}
		& \cdots & \xi_{i_0}^4& \cdots & \eta_{i_0}^4 \\
		1 & \xi_{i_1} & \eta_{i_1}
		& \cdots & \xi_{i_1}^4& \cdots & \eta_{i_1}^4 \\
		\vdots & \vdots  & \vdots &  \ddots & \vdots & \ddots & \vdots \\
		1 & \xi_{i_{|I_i|-1}} & \eta_{i_{|I_i|-1}}
		& \cdots & \xi_{i_{|I_i|-1}}^4& \cdots & \eta_{i_{|I_i|-1}}^4
	\end{pmatrix}
	\text{ and }
	\hat{\mathbf{b}} =
	\begin{pmatrix}
		u_h(z_{i_0})\\
		u_h(z_{i_1})\\
		\vdots \\
		u_h(z_{i_{|I_i|-1}})
	\end{pmatrix},
\end{equation*}
with $|I_i|$ being the cardinality of the set $I_i$.

To achieve the recovery gradient $G_hu_h$ at the basis vertex $z_i$, we first compute the geometric information of $\nabla p_{z_i}$ at $z_i$, that is
\[\mathcal G p_{z_i}^x(z_i):=\mathcal G\frac{\partial p_{z_i}}{\partial x}(z_i)=
        \begin{pmatrix}
          \hat a_2\\ \hat a_4 \\ \hat a_5 \\\hat a_8
        \end{pmatrix}\circ\bm l_1,\quad \mathcal G p_{z_i}^y(z_i):=\mathcal G\frac{\partial p_{z_i}}{\partial y}(z_i)=
        \begin{pmatrix}
          \hat a_3\\ \hat a_5 \\ \hat a_6 \\\hat a_9
        \end{pmatrix}\circ\bm l_2,\]
with \[\bm l_1 = \left(\frac1{h_{z_i}},\frac2{h_{z_i}^2},\frac1{h_{z_i}^2},\frac2{h_{z_i}^3}\right)^T,
\quad\bm l_2 = \left(\frac1{h_i},\frac1{h_i^2},\frac2{h_i^2},\frac2{h_i^3}\right)^T,\]
and $\circ$ represents the Hadamard product. Furthermore, the
interpolation coefficients at $z_i$ are determined by
\[\bm c_{z_i}^x = \mathbf B_i^{-1}\mathcal G p_{z_i}^x(z_i),\quad \bm c_{z_i}^y = \mathbf B_i^{-1}\mathcal G p_{z_i}^y(z_i)\]
Finally, the global recovered gradient can be assembled as
\begin{equation}\label{eq:Gh}
G_hu_h:=\left(\sum_{z_i}\bm c_{z_i}^x\cdot \bm \phi_{z_i},\,\, \sum_{z_i}\bm c_{z_i}^y\cdot \bm \phi_{z_i}\right)^T.
\end{equation}
The recovery procedure is summarized in Algorithm \ref{alg:ppr}.

%The local element patch $\Omega_{z_i}^n$ associated with $z_i$ is

%
%
\begin{algorithm}
    \caption{Superconvergent gradient recovery procedure}
    \label{alg:ppr}
    %\begin{algorithmic}[1]
    Given a T-mesh $\mathcal T_h$ and the PHT-spline solution $u_h$, repeat the following steps for all basis vertices $z_i \in \mathcal N_h$.\\
        (1)\,\, For every basis vertex $z_i$, construct a local patch of elements $\Omega_{z_i}$.
        Let $B_{z_i}$ be the set of basis vertices in $\Omega_{z_i}$ and $I_i$ be the indexes of the $B_{z_i}$.\\
        (2)\,\,Construct the reference local patch $\hat\Omega_{z_i}$ and the corresponding reference vertex set $\hat B_{z_i}$. \\
        (3)\,\,Find a polynomial $\hat p_{z_i}$ over $\hat\Omega_{z_i}$ by solving the least-squares problem:
        \[\hat p_{z_i} = \mathop{\text{arg min}}\limits_{\hat p}\sum_{j\in I_i}|\hat p(\hat z_{i_j}) - u_h(z_{i_j})|^2
        \quad\text{ for } \hat p \in\mathbb P_4(\hat\Omega_{z_i}).\]
        (4)\,\,Calculate the gradient of the approximated polynomial function
        \[\begin{pmatrix}p_{z_i}^x\\p_{z_i}^y\end{pmatrix}:=
        \nabla p_{z_i} =  \frac1{h_i}\hat\nabla\hat p_{z_i}\]
        (5)\,\, Calculate the geometric information of  $p_{z_i}^x$ and $p_{z_i}^y$
        \[\mathcal G p_{z_i}^x(z_i)=\mathcal G\frac{\partial p_{z_i}}{\partial x}(z_i)=
        \hat{\mathcal G}\frac{\hat \partial \hat p_{z_i}}{\hat \partial \hat x}(0,0)\circ\bm l_1,\quad
         \mathcal G p_{z_i}^y(z_i)=\mathcal G\frac{\partial p_{z_i}}{\partial y}(z_i)=
        \hat{\mathcal G}\frac{\hat \partial \hat p_{z_i}}{\hat \partial \hat y}(0,0)\circ\bm l_2\]
        (6)\,\, Calculate the interpolation coefficients
        \[\bm c_{z_i}^x = \mathbf B_i^{-1}\mathcal G p_{z_i}^x(z_i),\quad \bm c_{z_i}^y = \mathbf B_i^{-1}\mathcal G p_{z_i}^y(z_i)\]
        (7) Then, we have the recovered gradient at each vertex $z_i$
        \[G_hu_h:=\left(\sum_{z_i}\bm c_{z_i}^x\cdot \bm \phi_{z_i},\,\,\sum_{z_i}\bm c_{z_i}^y\cdot \bm \phi_{z_i}\right)^T\]
    %\end{algorithmic}
\end{algorithm}

\subsection{Properties of the PPR}\label{subsec:properties}
In this section, we analyze several crucial properties of our constructed PPR operator $G_h$ in \eqref{eq:Gh}.
From the construction of $G_h$, it is straightforward to observe that $G_h$ is a linear operator.
We begin by stating the following polynomial preservation property.
\begin{lemma}\label{lem:pp}
The recovery operator $G_h$ preserves quartic polynomial in the following sense:
for any basis vertex $z_i$, if $u$ is a quartic polynomial on the local patch $\Omega_{z_i}$, then we have $G_h(I_hu)=\nabla u$ on $\Omega_{z_i}$.
\end{lemma}
\begin{proof}
For any $u\in\mathbb P_4(\Omega_{z_i})$, we observe that $I_hu(z_{i_j})=u(z_{i_j})$, for all sampling points $z_{i_j} \in B_{z_i}$.
Therefore the least-squares fitting procedure exactly reproduces $u$. Then the desired result follows from the fact that $I_h$ is spline-preserving.
\end{proof}

Building upon the polynomial preservation property established above, we now proceed to analyze the boundedness of the operator $G_h$ and present the following lemma.

\begin{lemma}\label{lem:bd1}
For any basis vertex $z_i$, and for any function $u_h\in V_h$, we have
\begin{equation}\label{ineq:bd1}
 h|G_hu_h(z_i)|+ h^2\left|\frac{\partial G_hu_h}{\partial x}(z_i)\right|+h^2\left|\frac{\partial G_hu_h}{\partial y}(z_i)\right|+
 h^3\left|\frac{\partial^2 G_hu_h}{\partial x\partial y}(z_i)\right|
 \leq C|u_h|_{1,\Omega_{z_i}}.
\end{equation}
\end{lemma}
\begin{proof}
We first consider $G_hu_h(z_i)$. Based on the construction of $G_h$, the recovered gradient $G_hu_h$ at the point $z_i$ admits the representation:
\begin{equation}\label{eq:bd1-1}
  G_hu_h(z_i)=(
                G_h^xu_h(z_i),\,\,
                G_h^yu_h(z_i)
              )^T
              =\frac1{h_i}\left(
                \sum_{\ell=1}^{|I_i|}c_\ell^xu_h(z_{i_\ell}),\,\,
               \sum_{\ell=1}^{|I_i|}c_\ell^yu_h(z_{i_\ell})
              \right)^T,
\end{equation}
where $c_\ell^x$ and $c_\ell^y$  are independent of the mesh size. For constant function $\zeta\equiv u_h(z_i)$, according to
Lemma \ref{lem:pp}, we claim that
\begin{equation}\label{eq:bd1-2}
  G_h\zeta= (
                0,\,\,0
              )^T.
\end{equation}
Consequently, we deduce from \eqref{eq:bd1-1} and \eqref{eq:bd1-2} that
\begin{equation}\label{eq:bd1-3}
  G_hu_h(z_i)  =\frac1{h_i}\left(
                \sum_{\ell=1}^{|I_i|}c_\ell^x(u_h(z_{i_\ell})-u_h(z_i)),\,\,
               \sum_{\ell=1}^{|I_i|}c_\ell^y(u_h(z_{i_\ell})-u_h(z_i))
              \right)^T.
\end{equation}
For any $z_{i_\ell}$, we can find a sequence of vertices $z_i = z_{\ell_1}, z_{\ell_2}, \cdots , z_{\ell_{n_\ell}} = z_{i_\ell}$ such that the line segment $\overline{z_{\ell_j}z_{\ell_{j+1}}}=:s_{\ell_j}$ is an edge contained in an element $K\subset\Omega_{z_i}$. Then $u_h|_{s_{\ell_j}}$ is a polynomial, and it holds that
\begin{equation}\label{eq:bd1-4}
\frac{u_h(z_{i_\ell})-u_h(z_i)}{|s_{\ell_j}|}\leq |\nabla u_h|_{0,\infty,s_{\ell_j}}.
\end{equation}
By applying the inverse inequality and the trace inequality, we obtain
\begin{equation}\label{eq:bd1-5}
|\nabla u_h|_{0,\infty,s_{\ell_j}}\leq C|s_{\ell_j}|^{-\frac32}\|u_h\|_{0,s_{\ell_j}}\leq C(|s_{\ell_j}|^{-2}\|u_h\|_{0,K}+|s_{\ell_j}|^{-1}|u_h|_{1,K}).
\end{equation}
Combining the above estimates and noticing that $|s_{\ell_j}|/h_i$ is bounded by a fixed constant, we arrive at
\begin{equation}\label{eq:bd1-6}
\begin{split}
|G_hu_h(z_i)|&\leq C\sum_{K\subset\Omega_{z_i}}(h_i^{-2}\|u_h\|_{0,K}+h_i^{-1}|u_h|_{1,K})\\
&\leq C(h_i^{-2}\|u_h\|_{0,\Omega_{z_i}}+h_i^{-1}|u_h|_{1,\Omega_{z_i}}).
\end{split}
\end{equation}
Set $\bar u_h=\frac1{|\Omega_{z_i}|}\int_{\Omega_{z_i}}u_h(y)\,\mathrm dy$. Then, using Lemma \ref{lem:pp} again, we have
\[G_hu_h(z_i)=G_h(u_h-\bar u_h)(z_i).\]
Therefore, we obtain from \eqref{eq:bd1-6} and the Poincar\'e inequality that
\begin{equation}\label{eq:bd1-7}
|G_hu_h(z_i)|\leq C(h_i^{-2}\|u_h-\bar u_h\|_{0,\Omega_{z_i}}+h_i^{-1}|u_h-\bar u_h|_{1,\Omega_{z_i}})\leq Ch_i^{-1}|u_h|_{1,\Omega_{z_i}}.
\end{equation}

Similarly, the remaining three terms in \eqref{ineq:bd1} can be estimated in the same manner. For example, to estimate $\left|\frac{\partial G_hu_h}{\partial x}(z_i)\right|$, we only need to modify the representation in \eqref{eq:bd1-1} accordingly as follows:
\begin{equation}\label{eq:bd1-8}
\frac{\partial G_hu_h}{\partial x}(z_i)=
\frac1{h_i^2}\left(
\sum_{\ell=1}^{|I_i|}\tilde c_\ell^xu_h(z_{i_\ell}),\,\,
\sum_{\ell=1}^{|I_i|}\tilde c_\ell^yu_h(z_{i_\ell})
\right)^T,
\end{equation}
where $\tilde c_\ell^x$ and $\tilde c_\ell^y$ are constants independent of the mesh size.
\end{proof}

Based on the above lemma, we can establish the boundedness of the recovery operator $G_h$ in the $L^2$ norm.
\begin{lemma}\label{lem:bd2}
For any element $K \in \mathcal T_h$ and any function $u_h \in V_h$, there holds
\begin{equation}\label{ineq:bd2}
\|G_hu_h\|_{0,K} \leq C |u_h|_{1,\Omega_K},
\end{equation}
where $\overline{\Omega_K} := \overline{\bigcup_{z \in K \cap \mathcal N_h} \Omega_z}$. Furthermore, we have
\begin{equation}\label{ineq:bd2a}
\|G_hu_h\|_{0,\Omega} \leq C |u_h|_{1,\Omega}.
\end{equation}
\end{lemma}
\begin{proof}
  Let $I_K$ be the index set of $K\cap \mathcal N_h$, it is direct to check that
  \begin{equation}\label{ineq:bd2-1}
  \begin{split}
     \|G_hu_h\|_{0,K}&\leq C\sum_{\ell=1}^{|I_K|}\Bigg(|K|^{\frac12}G_hu_h(z_{i_\ell})+|K|^{\frac32}
     \left|\frac{\partial G_hu_h}{\partial x}(z_{i_\ell})\right|\\
     &\qquad\qquad+|K|^{\frac32}\left|\frac{\partial G_hu_h}{\partial y}(z_{i_\ell})\right|+|K|^{\frac52}
 \left|\frac{\partial^2 G_hu_h}{\partial x\partial y}(z_{i_\ell})\right|\Bigg)\\
 &\leq C\sum_{\ell=1}^{|I_K|}|u_h|_{1,\Omega_{z_{i_\ell}}}\leq C|u_h|_{1,\Omega_K}.
 \end{split}
  \end{equation}
  The estimate \eqref{ineq:bd2a} is a direct consequence of \eqref{ineq:bd2}, and the proof is completed.
\end{proof}

To conclude this section, we present the consistency result of $G_h$.
\begin{theorem}\label{thm:cons}
Suppose that $u \in H^5(\Omega_K)$, then the following estimate holds:
  \begin{equation}\label{ineq:cons}
    \|\nabla u-G_h(I_hu)\|_{0,K}\leq Ch^4\|u\|_{5,\Omega_K}.
  \end{equation}
\end{theorem}
\begin{proof}
Recalling the polynomial preserving property of $G_h$ from Lemma \ref{lem:pp}, the
conclusion follows directly by applying the Bramble--Hilbert lemma.
\end{proof}

\subsection{PPR methods for the general physical domain}\label{sec:phydomain}
In this section, we extend our PPR method to general physical domains. Suppose the physical domain $\Omega$ is parameterized by a global geometry
function $F:(s,t)\in\hat\Omega=(0,1)^2\to(x,y)\in\Omega$, defined as
\[F(s,t)=\sum_{i=1}^m\mathbf P_i\frac{w_i\phi_i(s,t)}{\sum_{j=1}^mw_j\phi_j(s,t)},\quad (s,t)\in\hat\Omega,\]
where $\mathbf P_i\in\mathbb R^2$, $\phi_i(s,t)$ is the PHT-splines basis function, $w_i>0$ is the weight, $m=4|\mathcal N_h|$
is the number of basis functions.

A finite dimensional subspace $\tilde V_h\subset H^1(\Omega)$ is defined as
\[
\tilde V_h=\{F_i(x,y):F_i(x,y)=\phi_i\circ F^{-1},\,\,i=1,2,\cdots,m\}.
\]
The isogeometric approximation of the weak form in \eqref{eq:var} is given as: Find $u_h\in \tilde V_h$ such that
\begin{equation}\label{eq:igadis}
B(u_h,v_h)=f(v_h)\quad\forall v_h\in\tilde V_h.
\end{equation}
The approximate solution $u_h$ can be written as
\[
u_h(x,y)=\sum_{i=1}^m c_i F_i(x,y)=\sum_{i=1}^m c_i \phi_i\circ F^{-1}(x,y),
\]
with unknown coefficients $c_i$, $i=1,2,\cdots,m$.

To obtain a higher-order numerical approximation of the gradient $\nabla u$, we employ the framework established in previous sections to construct the numerical gradient $G_hu_h$. To this end, for the solution $u_h$ of equation \eqref{eq:igadis}, we first perform a pull-back operation to the reference domain $\hat{\Omega}$, obtaining $\hat u_h$, which can be expressed as:
\[
    \hat u_h = \sum_{i=1}^m c_i \phi_i.
\]
For the spline function $\hat{u}_h$ on the reference domain $\hat{\Omega}$, we apply the PPR operator constructed in Section \ref{sec:recovery} to perform gradient recovery, obtaining the recovered numerical gradient $\hat G_h\hat u_h$ on the reference domain.
Define the Jacobian of the mapping $F$ as:
\[
\mathbf J = \begin{pmatrix}
\frac{\partial x}{\partial s} & \frac{\partial x}{\partial t} \\
\frac{\partial y}{\partial s} & \frac{\partial y}{\partial t}
\end{pmatrix}.
\]
Then, we push-forward $\hat G_h\hat u_h$ to the physical domain to obtain the recovered gradient $G_hu_h$ by
\[
G_hu_h = \mathbf J^{-T} \hat G_h \hat u_h \circ F^{-1}(x,y).
\]
Based on the same arguments presented in Section \ref{subsec:properties} and employing the standard estimation techniques in isogeometric analysis developed in \cite{BBCHS2006}, the consistency result stated in Theorem \ref{thm:cons} can be directly extended to general physical domains $\Omega$.

\subsection{Recovery-based a posteriori error estimator}\label{sec:posteriori}
A byproduct of the PPR operator is its utility in constructing \textit{a posteriori} error estimators. For any element $K\in\mathcal T_h$, we define
a local \textit{a posteriori} error estimator as
\begin{equation}\label{eq:estimatorlocal}
\eta_{h,K}:=\|G_hu_h-\nabla u_h\|_{0,K}
\end{equation}
and the corresponding global error estimator is given by
\begin{equation}\label{eq:estimatorglobal}\eta_h:=\left(\sum_{K\in\mathcal T_h}\eta_{h,K}^2\right)^{1/2}.\end{equation}

To measure the performance of the a posteriori error estimator \eqref{eq:estimatorlocal} or \eqref{eq:estimatorglobal},
we introduce the effective index
\[\kappa_h:=\frac{\|G_hu_h-\nabla u_h\|_{0,\Omega}}{\|\nabla u-\nabla u_h\|_{0,\Omega}}.\]
The \textit{a posteriori} error estimator \eqref{eq:estimatorlocal} or \eqref{eq:estimatorglobal} is said to be asymptotically
exact if \[\lim_{h\to0}\kappa_h=1.\]

\section{Superconvergence analysis on the translation invariant mesh}\label{sec:super}
In this section, we employ the superconvergence analysis via difference quotients, as developed in \cite{Wa1995}, to prove that the proposed gradient recovery method is superconvergent for translation invariant spline spaces.

We consider the following variational problem: Find $u\in H^1(\Omega)$ such that
\begin{equation}\label{eq:var}
  B(u,v)=\int_\Omega(\mathcal A\nabla u+\mathbf bu)\cdot\nabla v+cuv=f(v)\quad\forall v\in H^1(\Omega).
\end{equation}
Here $\mathcal A$ is a $2\times 2$ symmetric positive definite matrix, $\mathbf b$ is a vector, $c$ is a real number
and $f(\cdot)$ is a linear functional on $H^1(\Omega)$. For the sake of analytical convenience, all coefficient functions are assumed to be constant, and the physical domain $\Omega$ is taken as the unit square.

In order to insure \eqref{eq:var} admits a unique solution, we assume that the bilinear form $B(\cdot,\cdot)$ is bounded:
\[B(u,v)\leq C\|u\|_{1,\Omega}\|v\|_{1,\Omega}\quad \forall u,v\in H^1(\Omega),\]
and satisfies the inf-sup conditions (\cite{2006mixed,XZ2003}):
\[\inf_{u \in H^1(\Omega)} \sup_{v \in H^1(\Omega)} \frac{B(u, v)}{\|u\|_{1,\Omega}\|v\|_{1,\Omega}} = \sup_{u \in H^1(\Omega)} \inf_{v \in H^1(\Omega)} \frac{B(u, v)}{\|u\|_{1,\Omega}\|v\|_{1,\Omega}} \geq \gamma > 0.\]

The PHT-splines approximation of \eqref{eq:var} is to seek $u_h\in V_h$ satisfying
\begin{equation}\label{eq:varh}
  B(u_h,v_h)=f(v_h)\quad\forall v_h\in V_h.
\end{equation}

To insure a unique solution for \eqref{eq:varh}, we assume the discretization inf-sup conditions:
\[\inf_{u_h \in V_h} \sup_{v_h \in V_h} \frac{B(u, v)}{\|u\|_{1,\Omega}\|v\|_{1,\Omega}} = \sup_{u_h\in V_h} \inf_{v_h\in V_h} \frac{B(u, v)}{\|u\|_{1,\Omega}\|v\|_{1,\Omega}} \geq \tilde\gamma > 0.\]
From \eqref{eq:var} and \eqref{eq:varh}, it is obvious that the consistency property
\begin{equation}\label{eq:varcon}
  B(u-u_h,v_h)=0\quad\forall v_h\in V_h
\end{equation}
is true.

For any subdomain $\mathcal D\subset\Omega$, we denote $V_h(\mathcal D)$ as the restrictions of splines in $V_h$ to $\Omega$, and denote
$V_h^\comp(\Omega)$ as the set of those splines in $V_h(\mathcal D)$ with compact support in the
interior of $\mathcal D$ (\cite{Wa1995}). Let $\Omega_0\ssubset\Omega_1\ssubset\Omega_2\ssubset\Omega$ be separated by $d\geq c_0h$ with a generic
positive constant $c_0$, and $\ell$ be a direction, that is a unit vector in $\mathbb R^2$. Let $\tau$ be a parameter, which will typically be a
multiple of $h$. Let $T_\tau^\ell$ denote translation by $\tau$ in the direction $\ell$, namely,
\begin{equation}\label{eq:ttau}
  T_\tau^\ell v(x)=v(x+\tau\ell),
\end{equation}
and for an integer $\nu$, define
\begin{equation}\label{eq:ttau1}
  T_{\nu\tau}^\ell v(x)=v(x+\nu\tau\ell).
\end{equation}
Following the definition of \cite{Wa1995}, the PHT-splines space $V_h$ is called translation
invariant by $\tau$ in the direction $\ell$ if
\begin{equation}\label{eq:tran}
  T_{\nu\tau}^\ell v\in V_h^\comp(\Omega)\quad\forall v\in V_h^\comp(\Omega_1),
\end{equation}
for some integer $\nu$ with $|\nu|<M$. Equivalently, $\mathcal T_h$ is called a translation invariant
mesh. In particular, \eqref{eq:varcon} holds for any $v_h\in V_h^\comp(\Omega)$.

Our main theoretical analysis tool is superconvergence by difference quotient, as discussed in \cite{NS1974, Wa1995}. This is possible partially due to the interior error estimates for the finite element approximation, see \cite{NS1974} and \cite{SW1977,SW1995}.

The key observation is that $G_h$ can be viewed as a finite difference operator. As we have explained in previous section, the selection of sampling points in $\Omega_z$ is flexible. To adopt the argument of superconvergence by difference quotient, we choose the sampling points of the same type as the assembly point $z$, that is to say, we only select the basis vertex as the sampling points. It's essential to emphasize that this restriction is solely for theoretical purposes and not for computational purposes. %In Figure \ref{fig:sample}, we demonstrate two typical patterns in quadratic and cubic elements to illustrate how to choose sampling points of the same type in regularly patterned uniform meshes. Similar ideas are applicable to other translation-invariant meshes, such as chevron, criss-cross, union-Jack, and equilateral patterned uniform meshes.

Then, we can express the recovered gradient as
\begin{equation}\label{equ:diffq}
G_h^au_h(z)=\sum_{|\nu| \leq M} \sum_{i=1}^{n_\ell} C_{\nu, h}^i u_h\left(z+\nu h \ell_i\right) =
 \sum_{|\nu| \leq M} \sum_{i=1}^{n_\ell} C_{\nu, h}^iT_{\nu \tau}^{\ell_i} u_h(z),
\end{equation}
for $a\in \{x, y\}$, some natural numbers $M$, $n_\ell$,  and some direction $\ell_i$. Moreover, $C_{\nu, h}^i=\mathcal O(h^{-1})$.

Based on such an observation, we can prove the following lemma.
\begin{lemma}\label{lem:sup1}
let $\Omega_1\ssubset\Omega$ be separated by $d = O(1)$ and let the PHT-splines space $V_h$ be translation invariant in the directions required
by the gradient recovery operator $G_h$ on $\Omega_1$. Then on any interior domain $\Omega_0\ssubset\Omega_1$, if $u\in H^5(\Omega_1)$, we have
\begin{equation}\label{ineq:sup1}
  \|G_h(u-u_h)\|_{0,\Omega_0}\leq Ch^4\|u\|_{5,\Omega_1}+C\|u-u_h\|_{-s,q,\Omega_1},
\end{equation}
for some $s\geq 0$ and $q\geq1$.
\end{lemma}

\begin{proof}
Let $\Omega_0\ssubset\Omega_2\ssubset \Omega_1$ be separated by $O(1)$.
We start from the observation:
\begin{equation}\label{ineq:sup1-1}
  B(T_{\nu\tau}^\ell(u-u_h),v_h)=B(u-u_h,T_{-\nu\tau}^\ell v_h)=B(u-u_h,(T_{\nu\tau}^\ell)^*v_h)=0,
\end{equation}
for any $v_h\in V_h^\comp(\Omega_2)$. Then we claim by the fact $G_h$ is a difference operator constructed from translation of type \eqref{equ:diffq} that
\begin{equation}\label{ineq:sup1-2}
  B(G_h^x(u-u_h),v_h)=B(u-u_h,(G_h)^*v_h)=0 \quad\forall v_h\in V_h^\comp(\Omega_2).
\end{equation}
Next, we apply the interior error estimate result provided in \cite{NS1974} and \cite{SW1977}, yields
\begin{equation}\label{ineq:sup1-3}
  \|G_h^x(u-u_h)\|_{0,\Omega_0}\leq C\min_{v_h\in V_h}\|G_hu-v_h\|_{0,\Omega_2}+Cd^{-s-\frac2q}\|G_h^x(u-u_h)\|_{-s,q,\Omega_2}.
\end{equation}
For the first term in the righthand of above inequality, we employ the standard approximation theory to obtain:
\begin{equation}\label{ineq:sup1-4}
  \min_{v_h\in V_h}\|G_hu-v_h\|_{0,\Omega_2}\leq Ch^4\|u\|_{5,\Omega_1}.
\end{equation}
Next, we pay attention on the second term in the righthand of \eqref{ineq:sup1-3}.
we observe that
\begin{equation}\label{ineq:sup1-5}
  \|G_h^x(u-u_h)\|_{-s,q,\Omega_1}=\sup_{\varphi\in C_0^\infty(\Omega_2),\|\varphi\|_{s,q',\Omega_2}=1}(G_h^x(u-u_h),\varphi)
\end{equation}
with $\frac1q+\frac{1}{q'}=1$. Let $\Omega_2+Mh$ be a subdomain  stretches out $Mh$ from $\Omega_2$,
 we use the fact that $\|(G_h^x)^*\varphi\|_{0,1,\Omega_2+Mh}$ is bounded uniformly with respect to $h$ when $s\geq1$ to obtain
\begin{equation}\label{ineq:sup1-6}
  \begin{split}
     (G_h^x(u-u_h),\varphi)&=(u-u_h,(G_h^x)^*\varphi)\\
     &\leq C \|u-u_h\|_{0,\infty,\Omega_2+Mh}\|(G_h^x)^*\varphi\|_{0,1,\Omega_2+Mh}\\
     &\leq C \|u-u_h\|_{0,\infty,\Omega_2+Mh}.
  \end{split}
\end{equation}
Applying interior estimates in \cite{NS1974,SW1977} again, we get
\begin{equation}\label{ineq:sup1-7}
  \begin{split}
    \|u-u_h\|_{0,\infty,\Omega_2+Mh}&\leq C\min_{v_h\in V_h}\|u-v_h\|_{0,\Omega_1}+Cd^{-s-\frac2q}\|u-u_h\|_{-s,q,\Omega_1}\\
    &\leq Ch^4\|u\|_{5,\Omega_1}+Cd^{-s-\frac2q}\|u-u_h\|_{-s,q,\Omega_1}.
  \end{split}
\end{equation}
Since $d=O(1)$, we combine \eqref{ineq:sup1-3}, \eqref{ineq:sup1-4} and \eqref{ineq:sup1-7} to obtain
\[
  \|G_h^x(u-u_h)\|_{0,\Omega_0}\leq Ch^4\|u\|_{5,\Omega_1}+C\|u-u_h\|_{-s,q,\Omega_1}.
\]
Following the same argument, we can establish the same result for $G_h^y$, and complete the proof.
\end{proof}

We are now in a position to state our main result in this section.
\begin{theorem}\label{thm:insup}
  Under the same conditions as in the Lemma \ref{lem:sup1}, we have
  \begin{equation}\label{equ:sup}
		\|\nabla u - G_hu_h\|_{0, \Omega_0} \lesssim h^4\|u\|_{5, \Omega_1}+\|u-u_h\|_{-s, q, \Omega_1},
	\end{equation}
for any $s \geq 0$ and $q \geq 1$.
\end{theorem}

\begin{proof}
 To establish the superconvergence result, we decompose the error as
 \begin{equation}\label{equ:decomp}
 	\nabla u - G_hu_h = (\nabla u  - G_hu_I) + (G_hu_I-G_hu) + (G_hu - G_hu_h),
 \end{equation}
 where $u_I=I_hu$ and we shall estimate errors term by term  in  \eqref{equ:decomp}. The first term can be estimated by using the
 consistency of $G_h$ as in Theorem \ref{thm:cons}, namely,
 \begin{equation}\label{equ:firstt}
 	\|\nabla u  - G_hu_I\|_{0, \Omega_0} \leq h^4\|u\|_{5, \Omega_1}.
 \end{equation}
 The second term can be bounded by the approximation theory as
 \begin{equation}\label{equ:secondt}
	\|G_hu_I-G_hu\|_{0,\Omega_0}\leq h^4\|u\|_{5,\Omega_1}.
\end{equation}
For the last term, we have
\begin{equation}\label{equ:thirdt}
	\|G_h\left(u-u_h\right)\|_{0,\Omega_0} \leq h^4\|u\|_{5, \Omega_1}+\|u-u_h\|_{-s, q, \Omega_1},
\end{equation}
 by Lemma \ref{lem:sup1}. Substituting \eqref{equ:firstt} - \eqref{equ:thirdt} into \eqref{equ:decomp} gives
 the desired results.
 \end{proof}
\begin{remark}
Theorem \ref{equ:sup} is a superconvergence result under the condition
\[\|u-u_h\|_{-s, q, \Omega_1}\leq Ch^{3+\sigma},\,\,\sigma>0.\]
The reader is referred to \cite{NS1974} for negative norm estimates.
\end{remark}
\begin{remark}
It is noteworthy that Theorem \ref{thm:insup} does not require global regularity of the exact solution $u$, being valid under the weaker assumption of interior regularity only.
\end{remark}

\section{Numerical tests}\label{sec:num} In this section, we provide some numerical tests to illustrate our gradient
recovery method for the cubic PHT-splines. Our primary focus is on testing the superconvergence properties and the effectiveness of adaptive mesh refinement.
As described in Remark \ref{rmk:sp}, when selecting sample points, we can flexibly choose points within the element. Here, in addition to the basis vertices, we select an element center point and four Gauss integration points in each element, as shown in Fig. \ref{fig:sp}.
\begin{figure}[ht]
\centering
\includegraphics[width=0.36\textwidth]{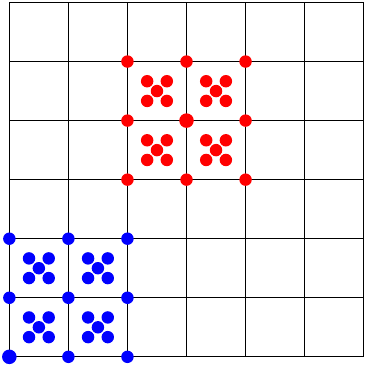}
\caption{Sample points for the gradient recovery.}
\label{fig:sp}
\end{figure}

\subsection{Superconvergence tests on the unit square} We present several different tests to verify our superconvergence results stated in Theorem \ref{thm:insup}.
We shall consider the following Poisson equation:
\begin{equation}\label{eq:model}
\begin{cases}
-\Delta u=f &\text{in } \Omega,\\
u = g &\text{on } \partial\Omega.
\end{cases}
\end{equation}
In general, the quality of $G_hu_h$ deteriorates near the boundary of $\Omega$. Therefore, we should study
the behavior of $G_hu_h$ inside $\Omega$. To distinguish between the
domains inside $\Omega$ and the ones adjacent to $\partial\Omega$, $\mathcal N_h$ is partitioned into $\mathcal N_{h,1}\cup\mathcal N_{h,2}$ where
\[\mathcal N_{h,1}=\{z\in \mathcal N_h:\text{dist}(z,\partial\Omega)\geq L\},\]
for a given positive constant $L$, in the following examples we choose $L$ = 0.125. Then we can define
\[\Omega_{h,1}=\{K\in\mathcal T_h: K \text{  has all of its vertices in }\mathcal N_{h,1}\}.\]

\textit{Test 1.} In this test, we solve the equation \eqref{eq:model} on a uniform tensor product mesh over
the unit square $\Omega=(0,1)\times(0,1)$ using PHT-splines. The exact solution is taken as $u=\sin(\pi x)\sin(\pi y)$ and $f$ is determined by $u$.
We numerically compare the errors of $\|\nabla u-\nabla u_h\|_{0,\Omega}$ and $\|\nabla u-G_hu_h\|_{0,\Omega_{h,1}}$, and plot the convergence rates in Fig.\ref{fig:test1}.

From Fig. \ref{fig:test1}, it can be observed that the convergence rate of the error in the energy norm is $\mathcal{O}(h^3)$, which is consistent with the known theoretical results. Moreover, the error of the recovered gradient by the PPR operator exhibits a superconvergent rate of $\mathcal{O}(h^4)$, which agrees well with the prediction given in Theorem \ref{thm:insup}.

\begin{figure}[ht]
    \centering
    \subfigure{
        \includegraphics[width=0.45\textwidth]{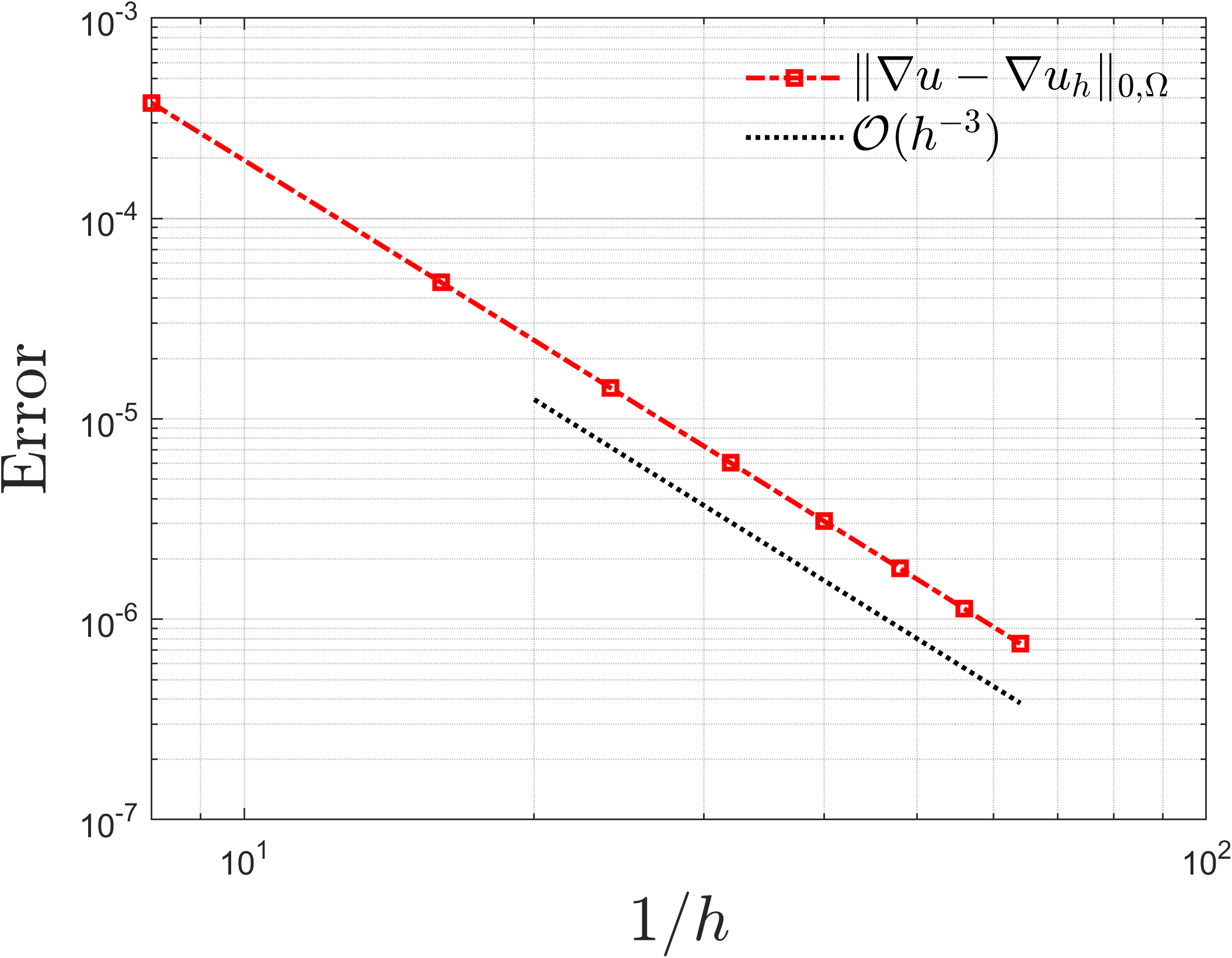}
    }
    \subfigure{
	\includegraphics[width=0.45\textwidth]{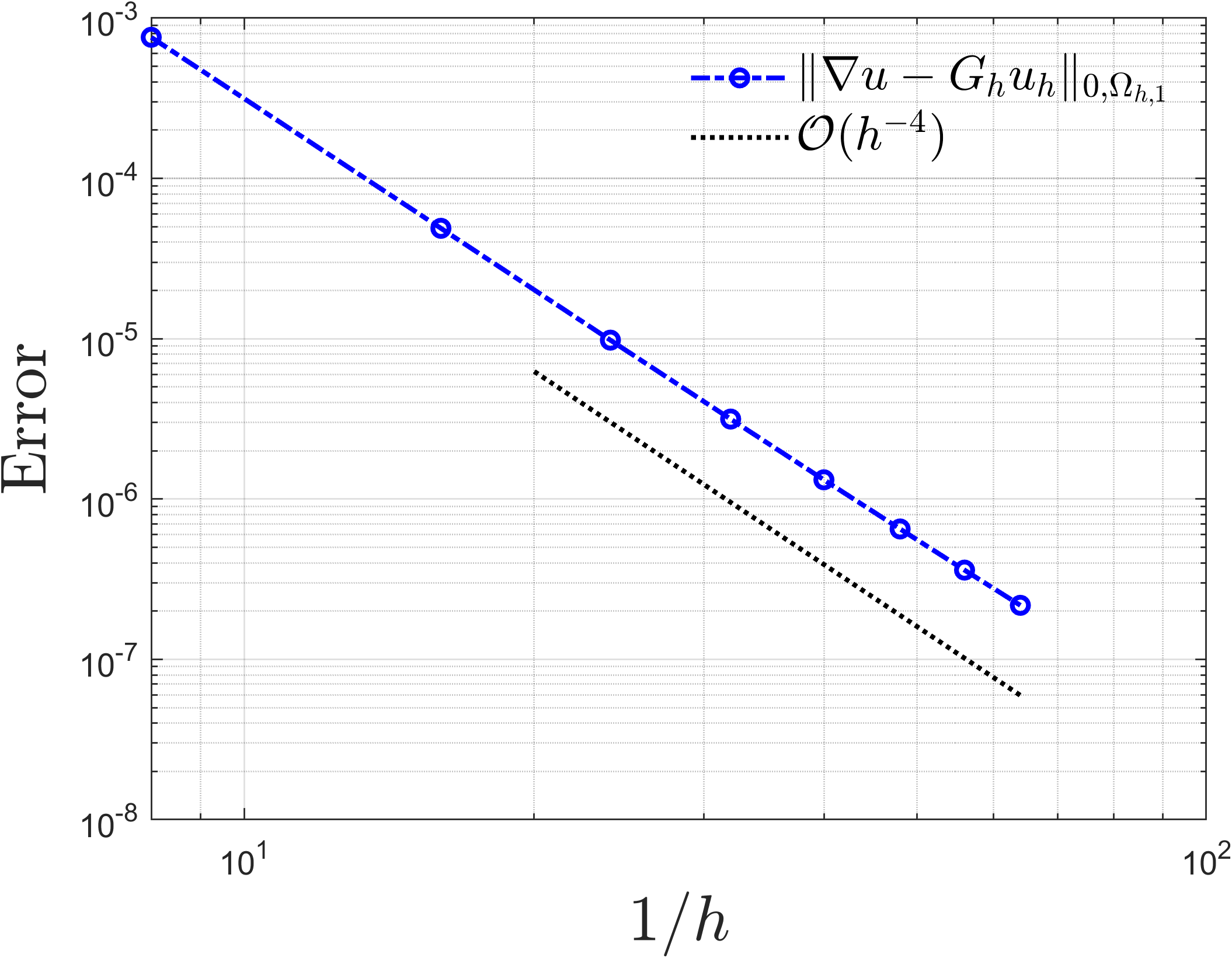}
    }
\caption{ Convergence behavior of numerical errors in Test 1. }
    \label{fig:test1}
\end{figure}

\textit{Test 2.}  In this test, we consider a problem with nonhomogeneous Dirichlet boundary conditions. The exact solution is taken as $u = x^7y^5$, and the right-hand side function $f$ is determined accordingly. The numerical results are presented in Fig.\ref{fig:test2}. From the figure, we observe that the convergence rates for both the energy norm error $\|\nabla u-\nabla u_h\|_{0,\Omega}$ and the recovered gradient error $\|\nabla u-G_hu_h\|_{0,\Omega_{h,1}}$ are consistent with those in Test 1, namely $\mathcal O(h^3)$ and $\mathcal O(h^4)$, respectively. This further confirms the theoretical prediction stated in Theorem \ref{thm:insup}.

\begin{figure}[ht]
    \centering
    \subfigure{
        \includegraphics[width=0.45\textwidth]{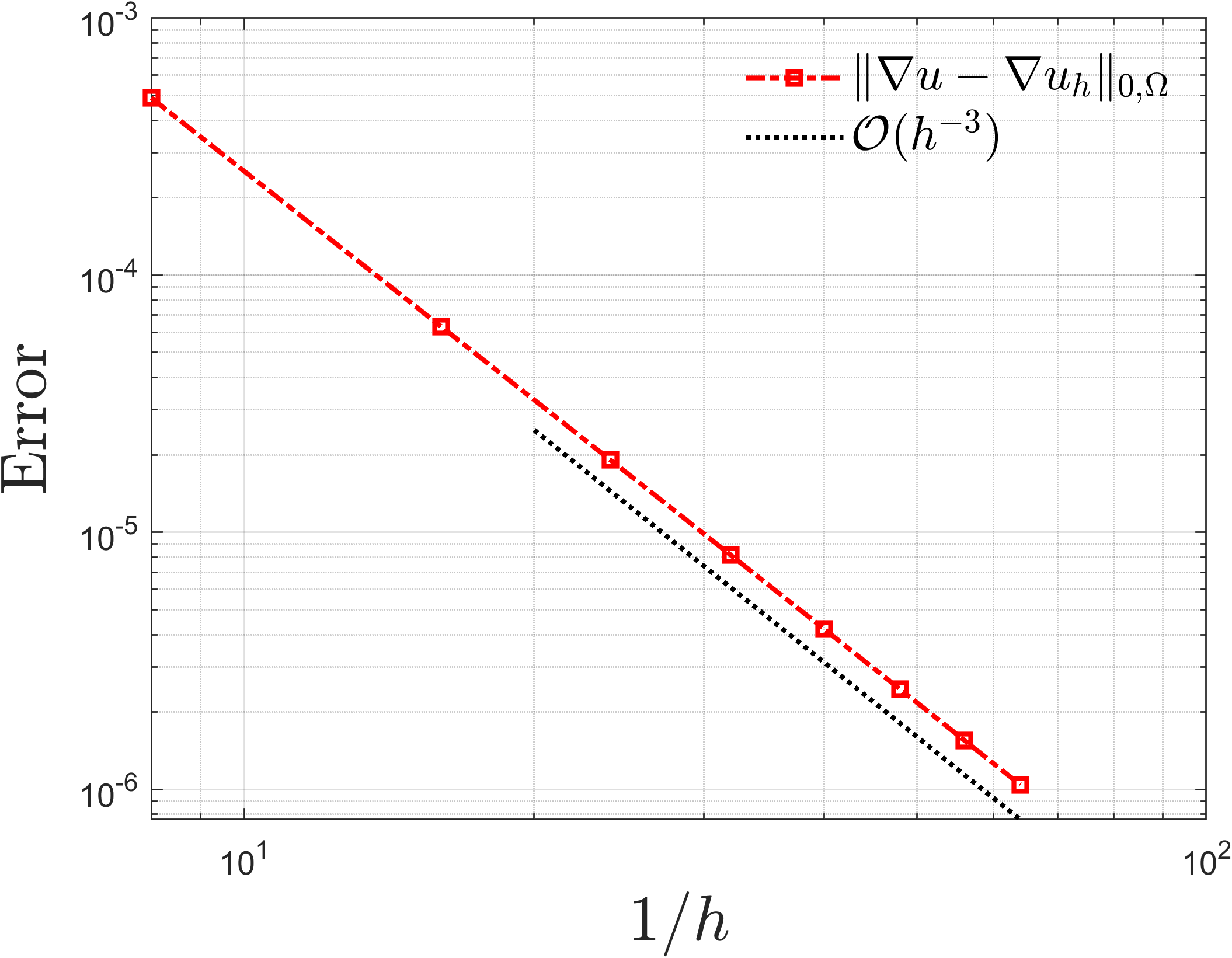}
    }
    \subfigure{
	\includegraphics[width=0.45\textwidth]{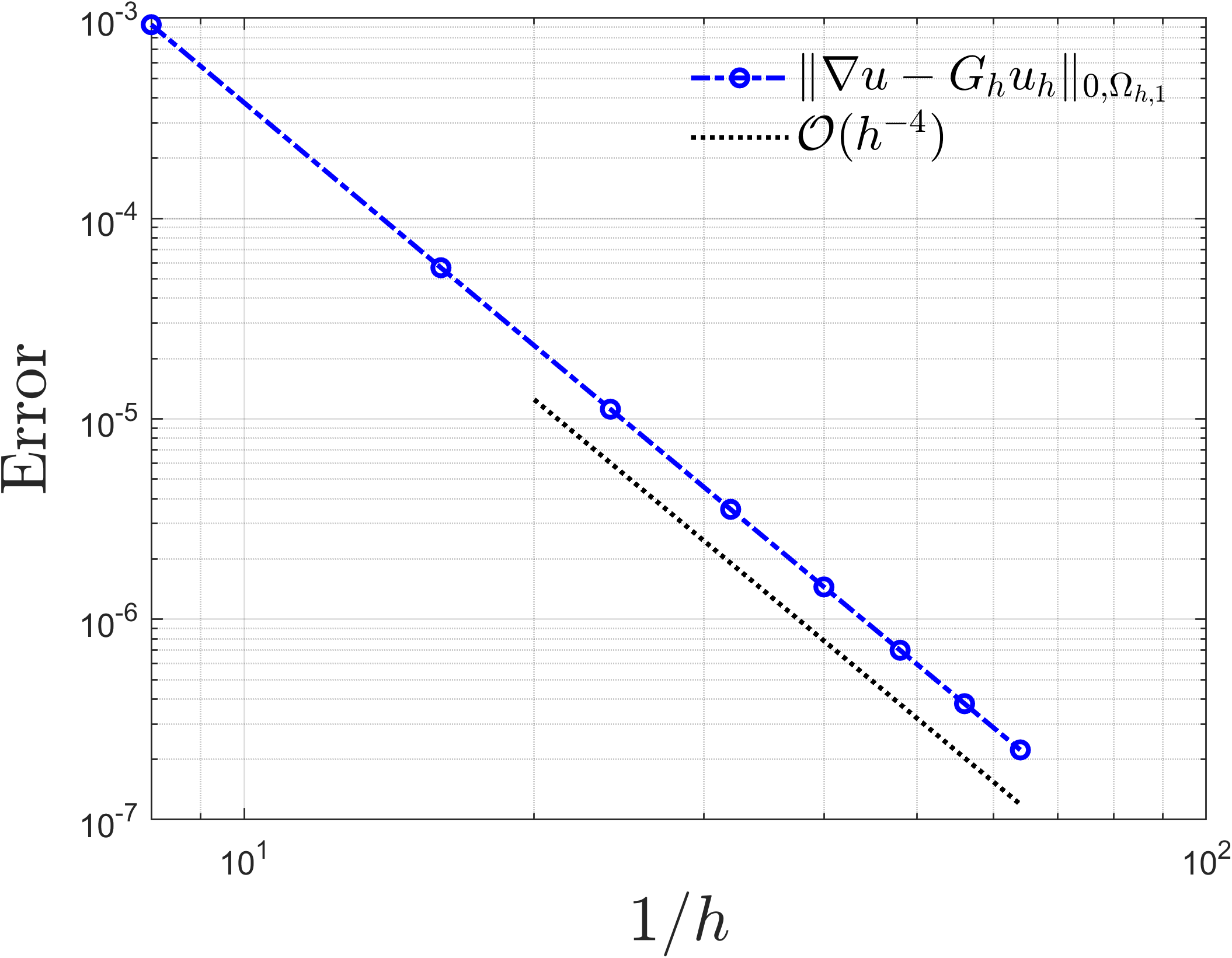}
    }
\caption{Error}
    \label{fig:test2}
\end{figure}

\subsection{Tests for the general physical domains}
\textit{Test 3.} In this test, we solve a similar problem as considered in \cite{KKJ2017}. The computational domain is a quarter annulus defined in polar coordinates as
\[
\Omega = \{ (\rho, \varphi) : \frac{1}{2} < \rho < 1,\ 0 < \varphi < \frac \pi2 \},
\]
and illustrated in Fig.~\ref{fig:test3} (left). The exact solution is given in Cartesian coordinates by
\[
u(x,y) = (x^2 + y^2 - \tfrac{1}{4})(x^2 + y^2 - 1)\sin(x)\sin(y),
\]
from which the right-hand side $f$ and boundary conditions are determined accordingly.

The convergence behavior of the numerical and recovered gradient errors is presented in Fig. \ref{fig:test3} (right). We observe that the discrete gradient error $\|\nabla u - \nabla u_h\|_{0,\Omega}$ converges at a rate of $\mathcal O(h^3)$, while the recovered gradient error
$\|\nabla u - G_h u_h\|_{0,\Omega_{h,1}}$ exhibits a superconvergent rate of $\mathcal O(h^4)$. These results demonstrate that the proposed PPR method retains its high accuracy and superconvergence property even on curved and more general physical domains, consistent with the findings in Tests 1 and 2 on the unit square.

\begin{figure}[ht]
    \centering
    \subfigure{
        \includegraphics[width=0.45\textwidth]{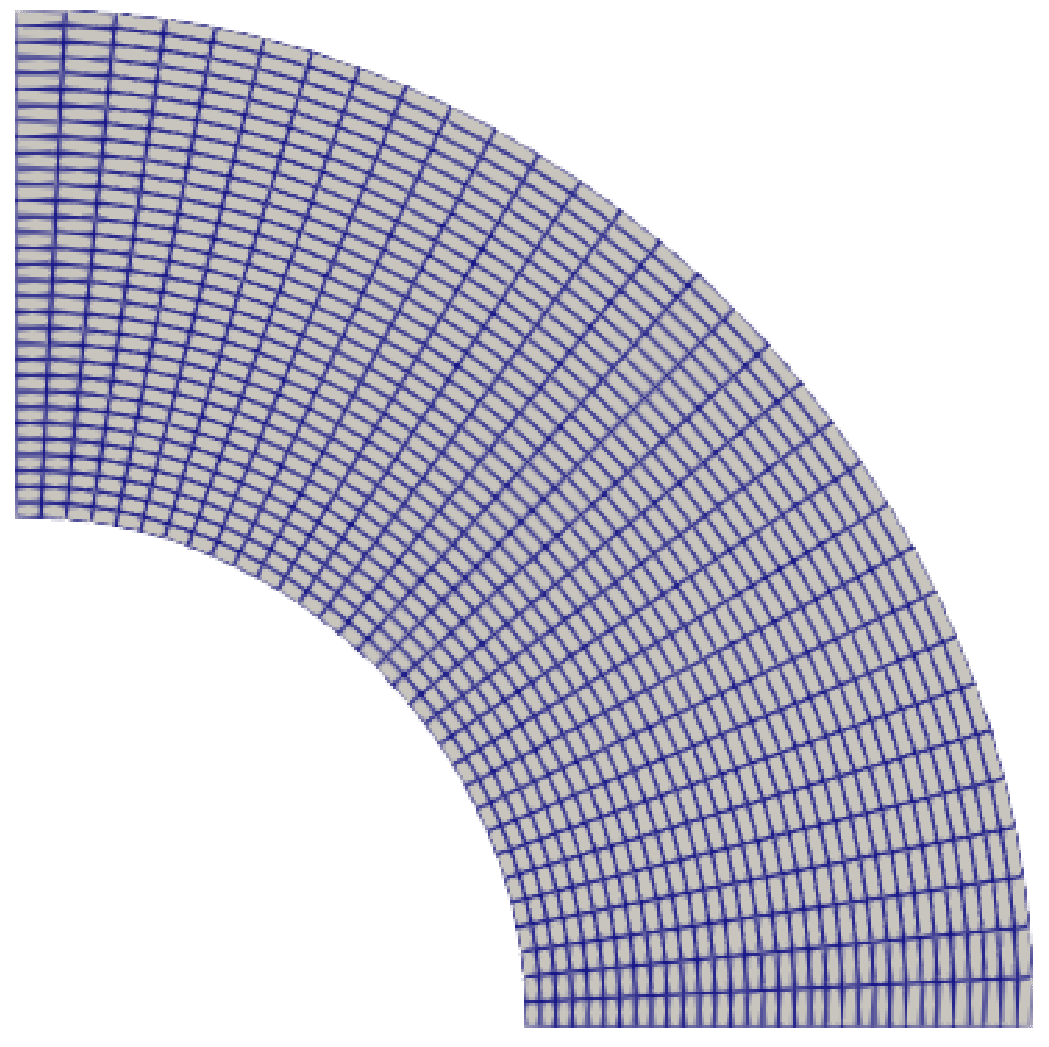}
    }
    \subfigure{
	\includegraphics[width=0.45\textwidth]{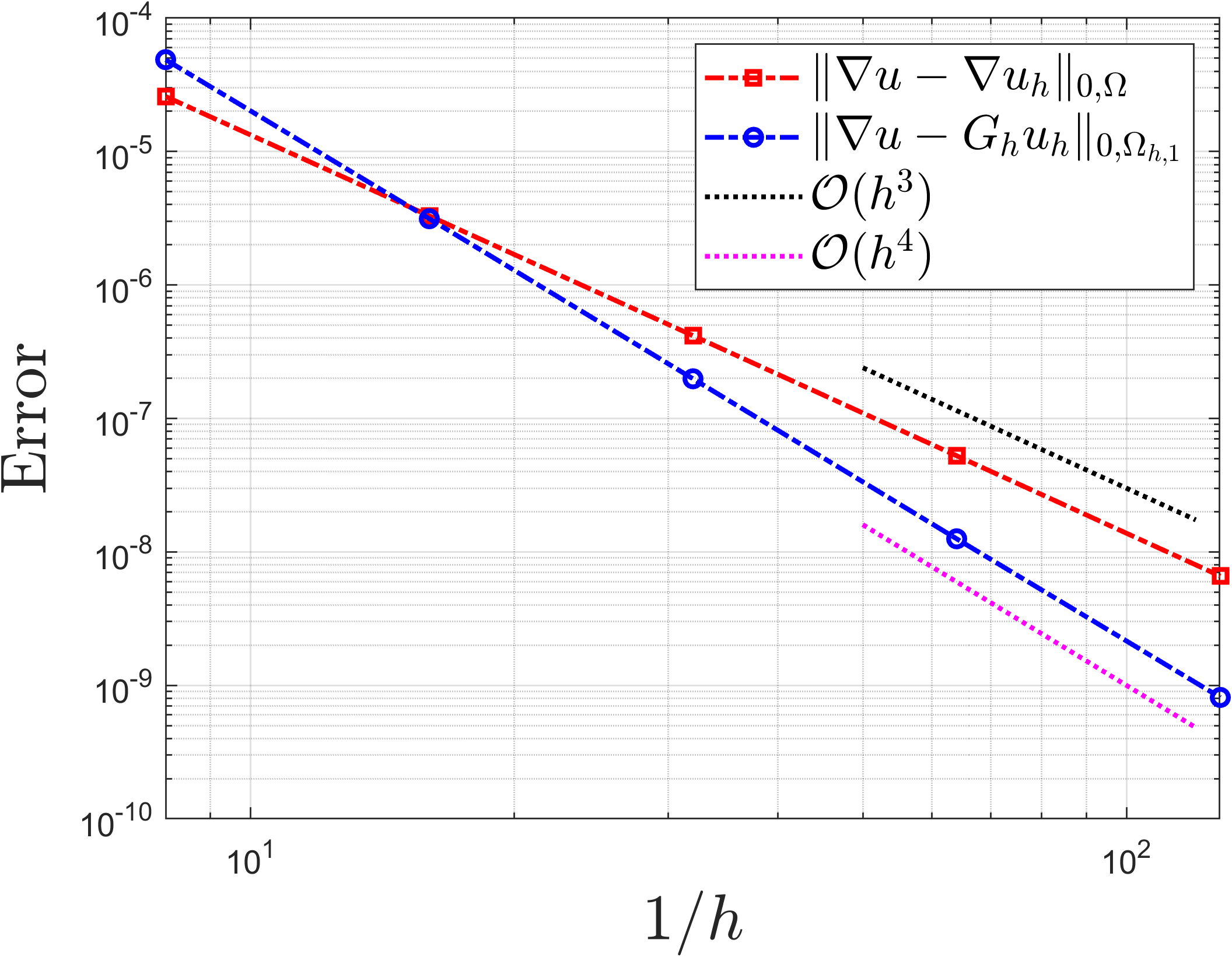}
    }
\caption{Illustration of the physical domain (left) and convergence behavior of numerical and recovered gradient errors (right) in Test 3.}
    \label{fig:test3}
\end{figure}

\textit{Test 4.} This test is similar to Test 3, but conducted on a different general physical domain. Specifically, the computational domain is taken as the unit circle centered at the origin. The exact solution is chosen as
\[
u(x,y) = x^7y^5,
\]
from which the source term $f$ and the nonhomogeneous Dirichlet boundary condition are computed accordingly.

Fig. \ref{fig:test4} (left) illustrates the circular physical domain, while the convergence behavior of the numerical and recovered gradient errors is shown in Fig. \ref{fig:test4} (right). As in Test 3, we observe that the numerical gradient error $\|\nabla u - \nabla u_h\|_{0,\Omega}$ converges at a rate of $\mathcal O(h^3)$, and the recovered gradient error $\|\nabla u - G_h u_h\|_{0,\Omega_{h,1}}$ exhibits a superconvergent rate of $\mathcal O(h^4)$. These results further confirm that the proposed PPR method is effective for general physical domains, and that the superconvergence properties observed on structured domains (Tests 1 and 2) still hold in more complex geometries.

\begin{figure}[ht]
    \centering
    \subfigure{
        \includegraphics[width=0.45\textwidth]{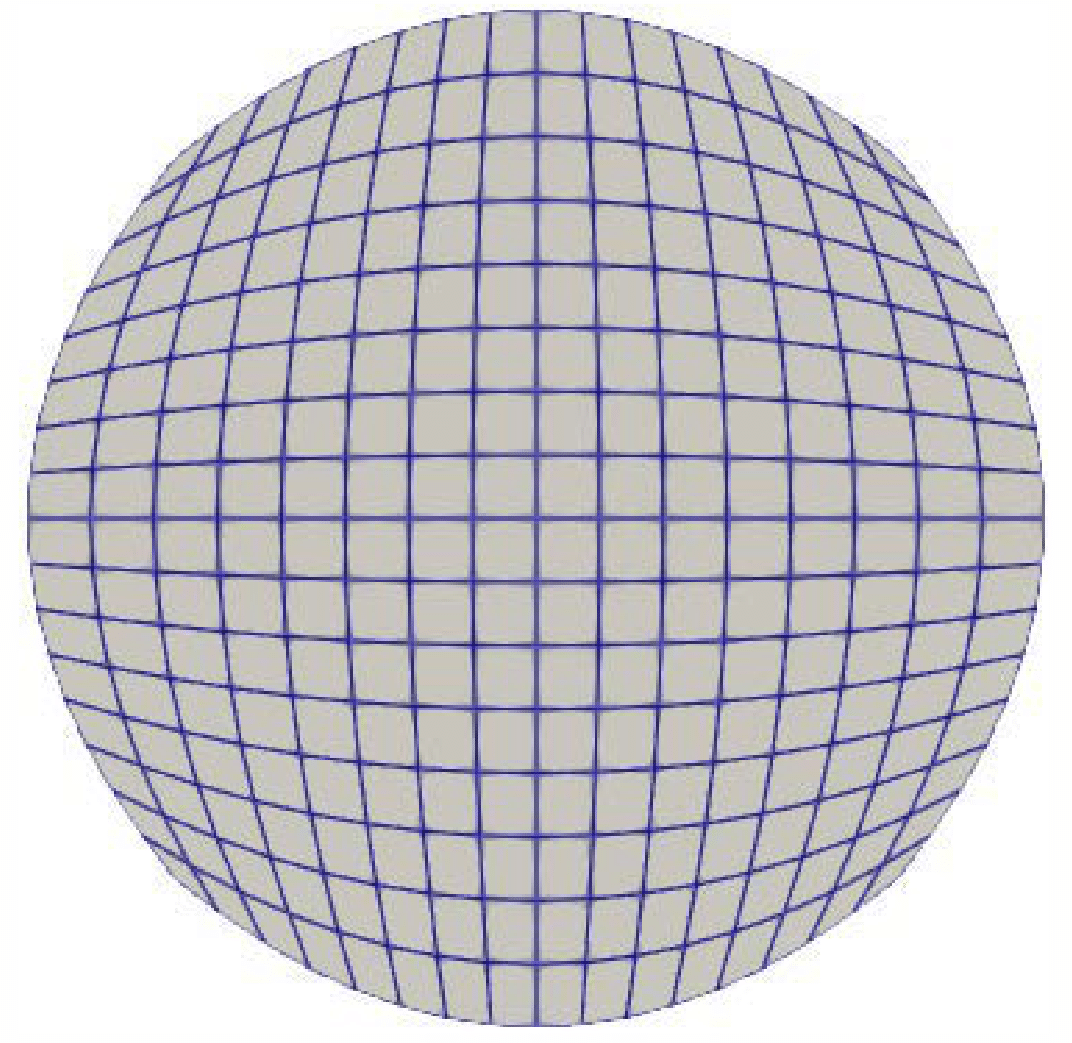}
    }
    \subfigure{
	\includegraphics[width=0.45\textwidth]{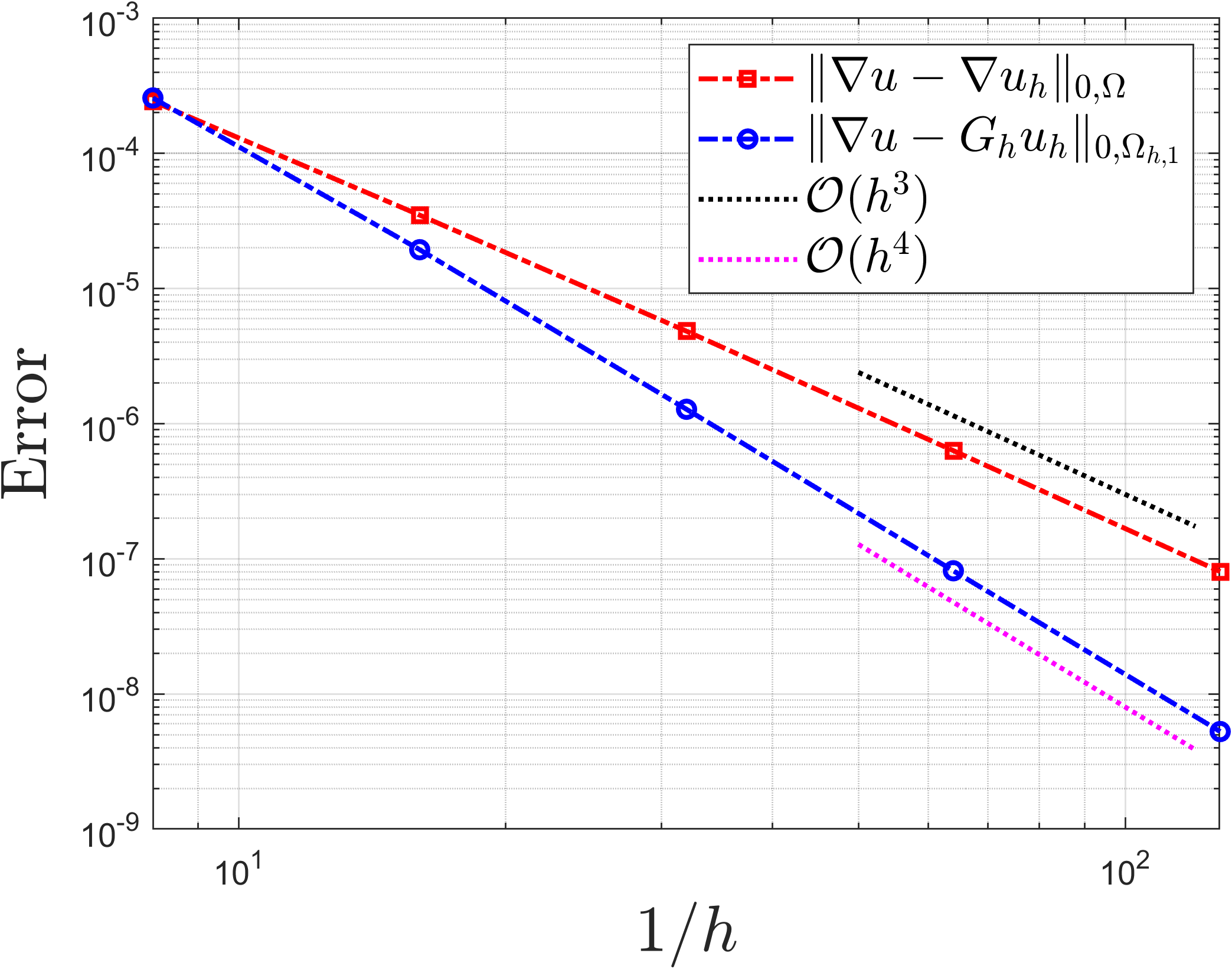}
    }
\caption{Illustration of the physical domain (left) and convergence behavior of numerical and recovered gradient errors (right) in Test 4.}
    \label{fig:test4}
\end{figure}

\subsection{Adaptive method tests}

This section is devoted to testing the performance of the proposed recovery-type \textit{a posteriori} error estimator \eqref{eq:estimatorglobal}. We apply adaptive mesh refinement to problems with localized or directional singularities and examine the convergence behavior of both the numerical and recovered gradients.

\textit{Test 5}.  In this test, we solve the Poisson equation \eqref{eq:model} on the unit square domain $\Omega = (0,1)^2$. The exact solution is given by
\[u=\frac1{(x-0.5)^2+(y-0.5)^2+0.02},\]
from which the corresponding source term $f$ and the Dirichlet boundary condition are computed. Note that the exact solution exhibits a sharp peak at the center point $(0.5, 0.5)$, introducing a localized singularity. To accurately capture this singular behavior, we apply the adaptive PHT-splines method.

In Fig. \ref{fig:test5}(left), we plot the convergence rates of the discrete $H^1$ semi-norm error $\|\nabla u-\nabla u_h\|_{0,\Omega}$ and the recovery error $\|\nabla u-G_h u_h\|_{0,\Omega}$. %For a uniformly partitioned tensor-product grid, according to the dimension formula \eqref{eq:dimformula}, the number of degrees of freedom (DOFs) is given by $4(h_l^{-1}+1)^2$, where $h_l$ is the mesh size. Consequently, we have $\text{DOFs}=\mathcal O(h_l^{-2})$.
%Therefore, in the context of adaptive mesh refinement, we consider the convergence rate $\mathcal{O}(N^{-1})$ with respect to the number of degrees of freedom $N$ as optimal.
It is clearly observed that the adaptive refinement successfully recovers the optimal convergence behavior of the numerical gradient, while the recovered gradient exhibits a superconvergent rate of $\mathcal O(h^3)$. This demonstrates the robustness and high accuracy of the PPR-based \textit{a posteriori} error estimator.

To further assess the performance of the estimator, we display the effective index in Fig. \ref{fig:test5}(right). The effective index rapidly approaches 1 as the mesh is refined, demonstrating that the proposed a posteriori error estimator is asymptotically exact.
\begin{figure}[ht]
    \centering
	\includegraphics[width=0.7\textwidth]{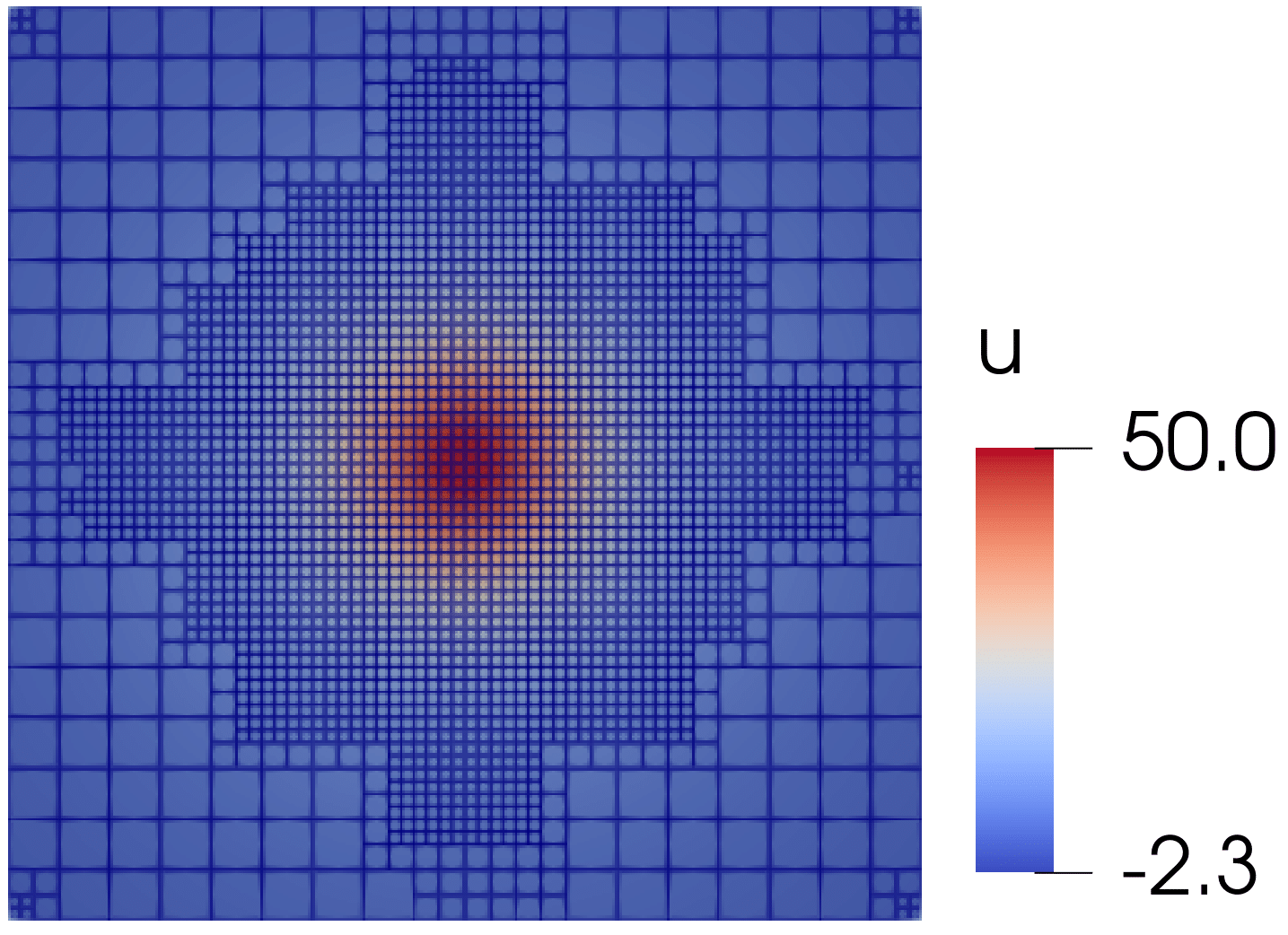}
\caption{Adaptively refined mesh for Test 5.}
    \label{fig:test5mesh}
\end{figure}

\begin{figure}[ht]
    \centering
    \subfigure{
        \includegraphics[width=0.45\textwidth]{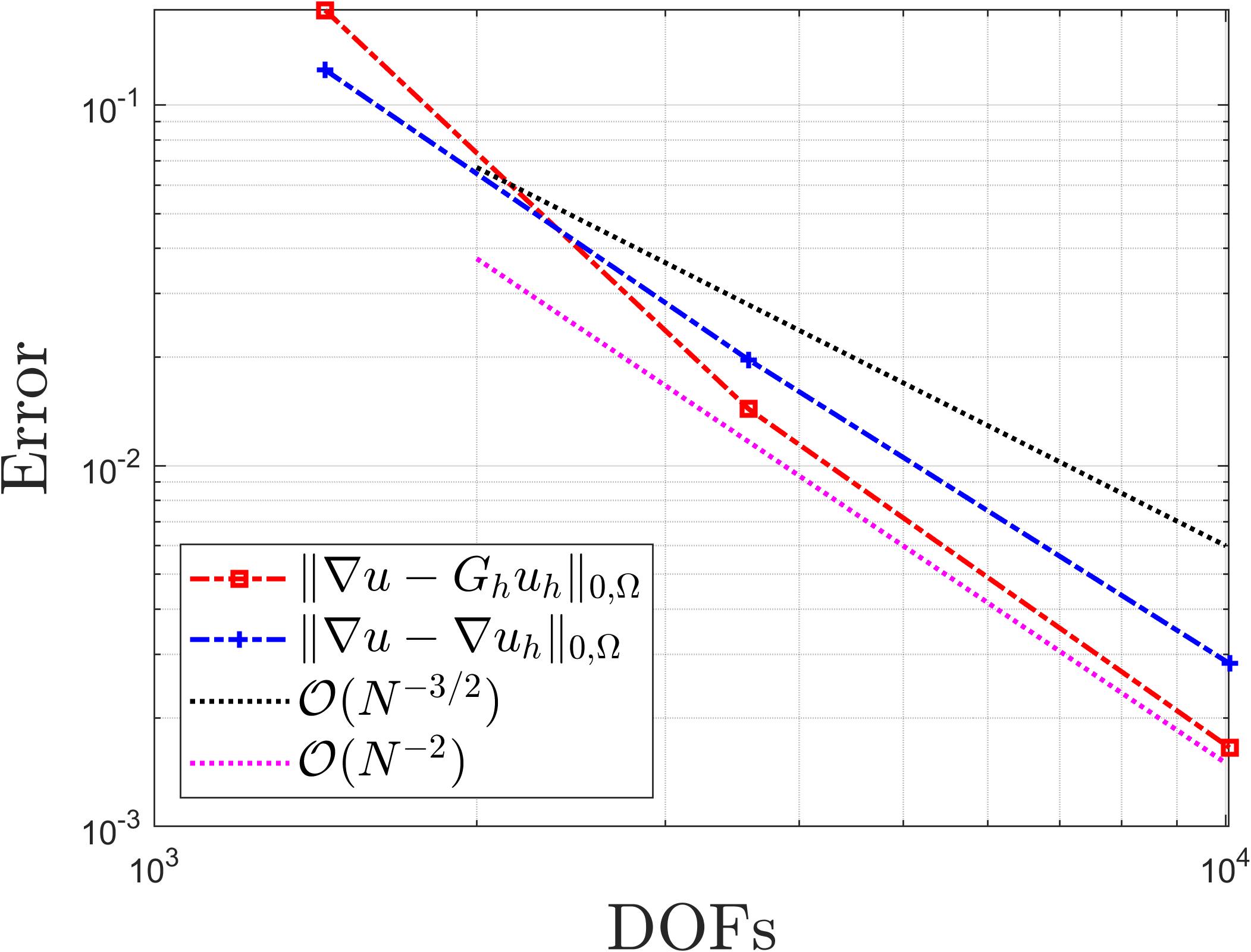}
    }
    \subfigure{
	\includegraphics[width=0.45\textwidth]{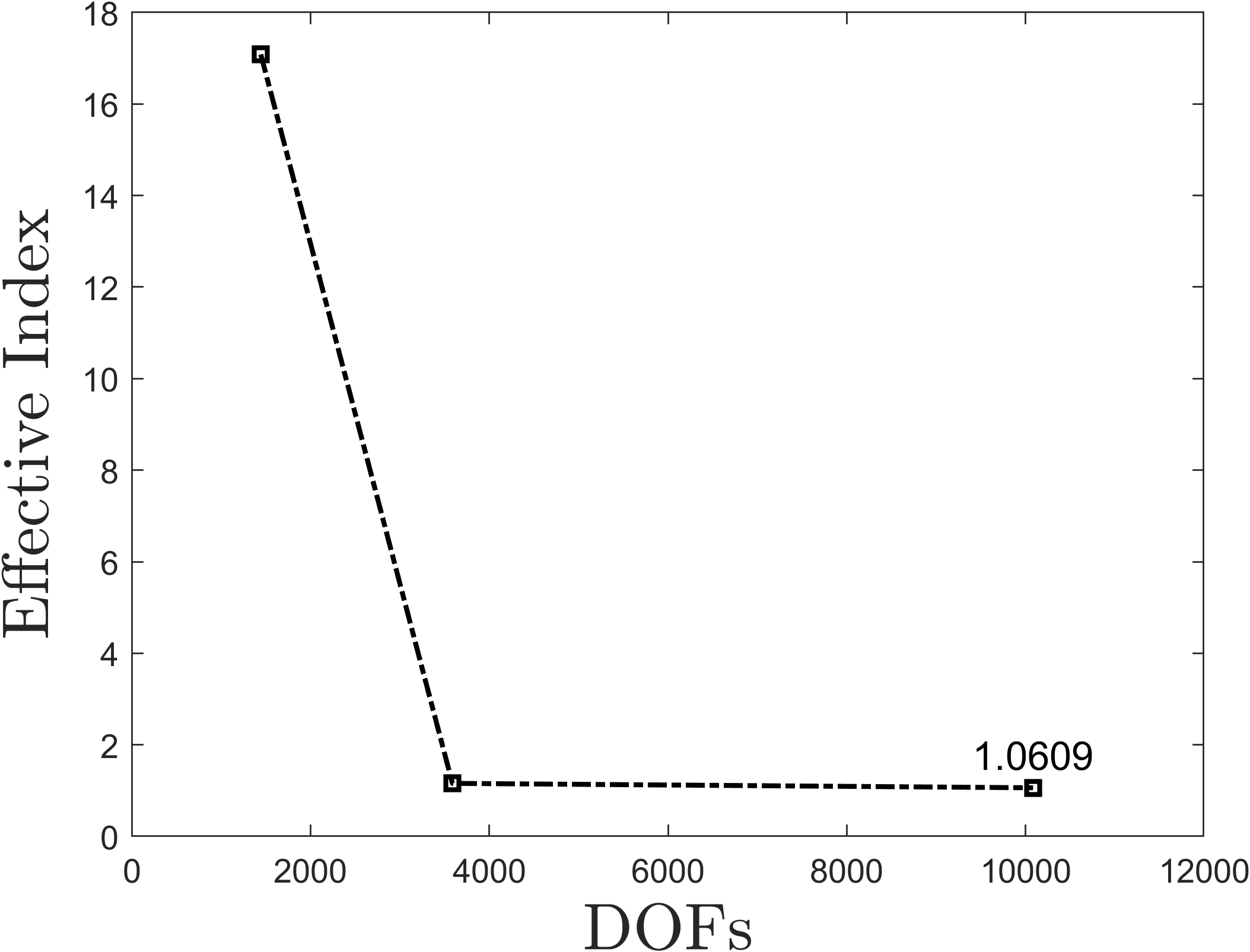}
    }
\caption{Convergence and effectivity results for the adaptive PHT-splines method in Test 5.}
    \label{fig:test5}
\end{figure}
\textit{Test 6}. As in \cite{CGPS2017,GXZ2019}, we consider the Poisson equation \eqref{eq:model} on the unit square $\Omega = (0,1)^2$ with a sharp interior layer. The exact solution is given by
\[u=16x(1-x)y(1-y)\arctan(25x-100y+25),\]
which exhibits a strong variation along a diagonal-like interior layer. This introduces a directional singularity that poses challenges for uniform mesh refinement. Similar to Test 5, we apply the adaptive PHT-splines method to resolve the local singular behavior.

In Fig. \ref{fig:test6mesh}, we show the corresponding adaptively refined mesh. The numerical results are shown in Fig. \ref{fig:test6}. The convergence rates of the discrete gradient error and the recovered gradient error are observed to follow the same trend as in Test 5, demonstrating optimal and superconvergent behavior, respectively. Notably, since the \textit{a posteriori} error estimator is asymptotically exact, our numerical results reveal that the recovered gradient achieves higher accuracy than the discrete gradient when the mesh is sufficiently refined. The corresponding effective index also converges rapidly towards 1, further confirming the accuracy and asymptotic exactness of the recovery-based \textit{a posteriori} error estimator.

\begin{figure}[ht]
    \centering
	\includegraphics[width=0.7\textwidth]{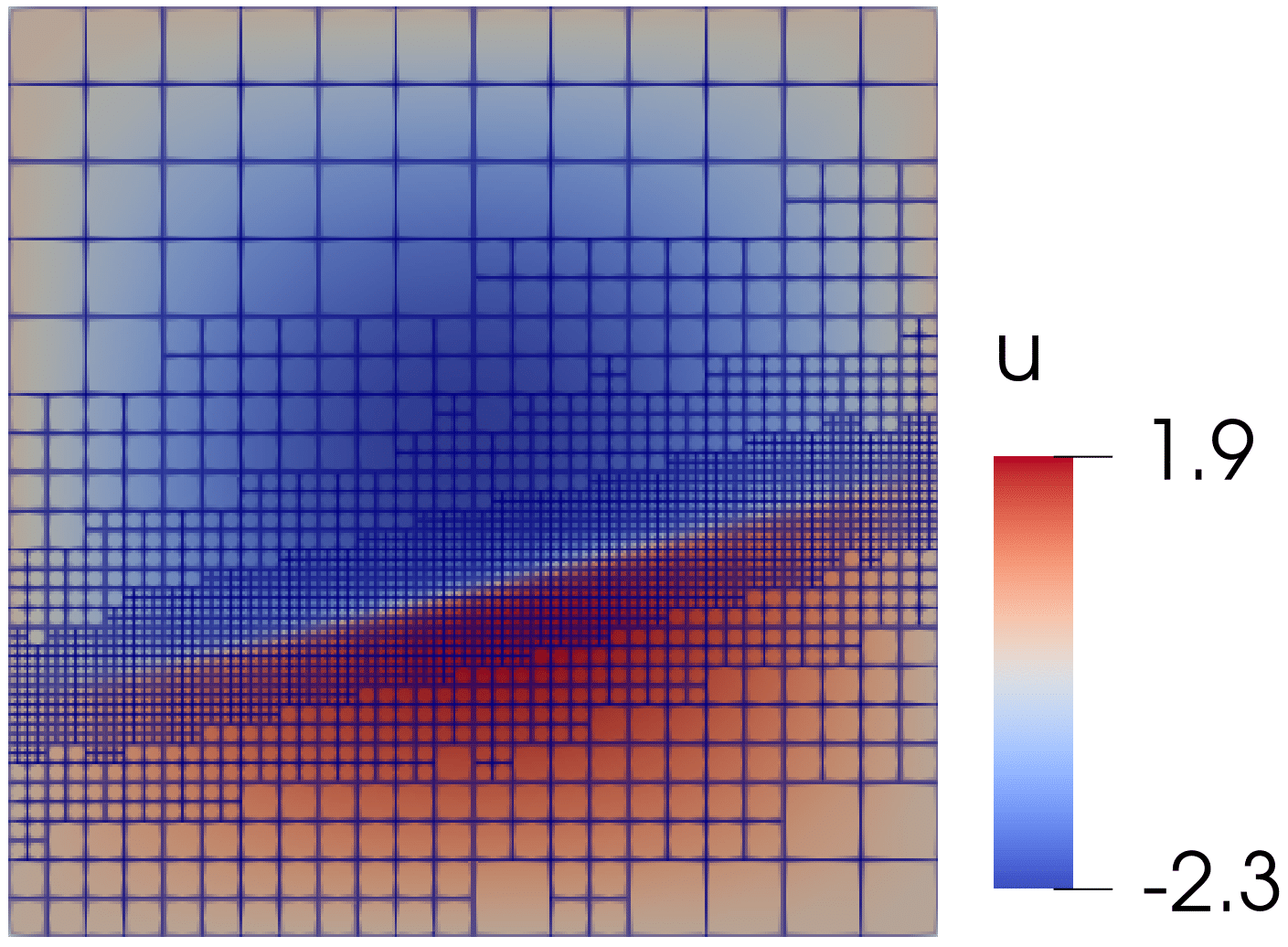}
\caption{Adaptively refined mesh for Test 6.}
    \label{fig:test6mesh}
\end{figure}

\begin{figure}[ht]
    \centering
    \subfigure{
        \includegraphics[width=0.45\textwidth]{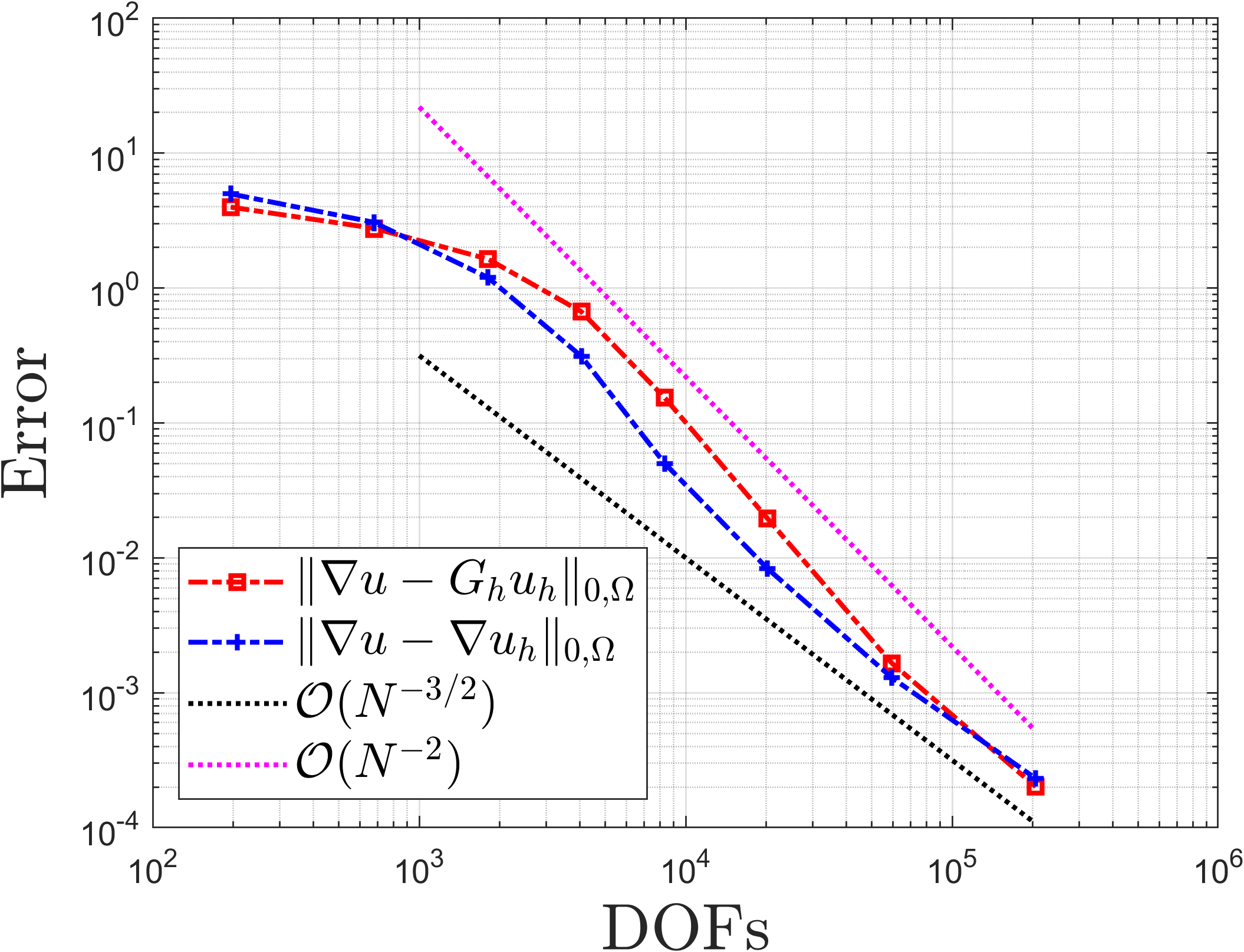}
    }
    \subfigure{
	\includegraphics[width=0.45\textwidth]{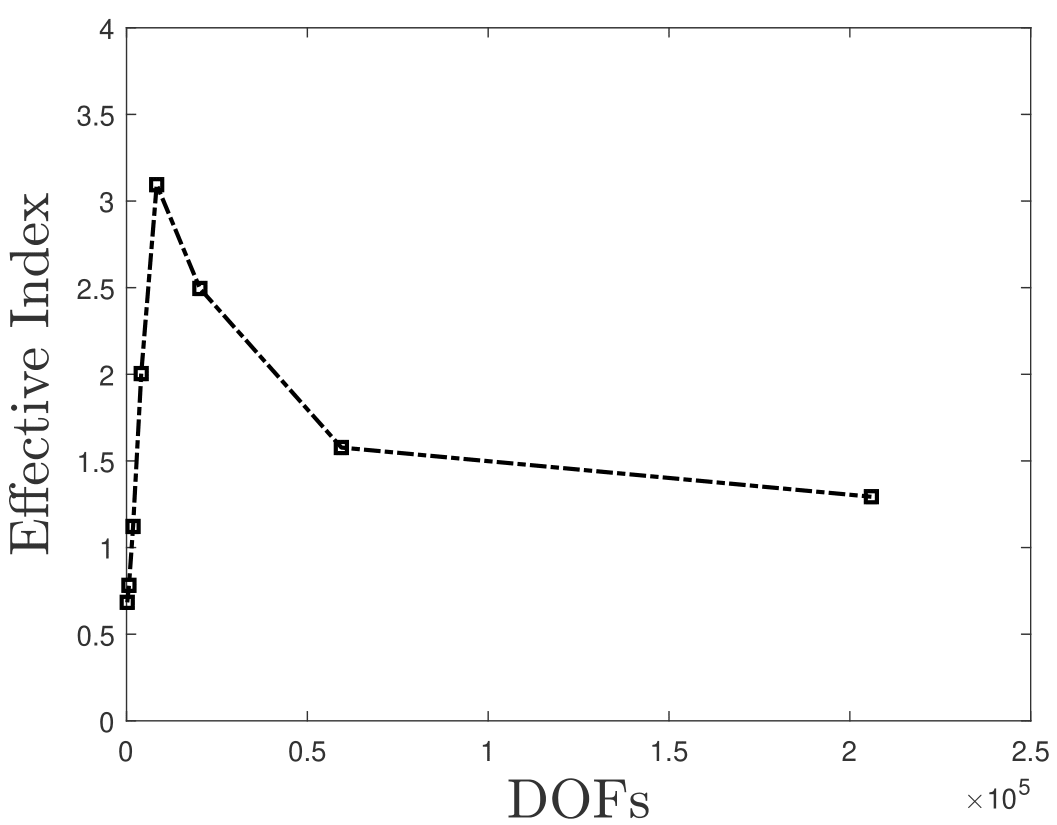}
    }
\caption{Convergence and effectivity results for the adaptive PHT-splines method in Test 6.}
    \label{fig:test6}
\end{figure}

\section{Conclusion}
In this paper, we have developed a polynomial preserving recovery method for PHT-splines in isogeometric analysis. The method takes advantage of the interpolation structure of PHT-splines and offers flexibility in sampling point selection. We established the superconvergence of the recovered gradient on translation-invariant meshes using classical analytical tools. In addition, a recovery-based \textit{a posteriori} error estimator was proposed for adaptive refinement.

\section*{Authorship contribution statement}
Ying Cai: Writing – original draft, Writing – review \& editing; Falai Chen: Writing – original draft, Writing – review \& editing; Hailong Guo: Writing – original draft, Writing – review \& editing; Hongmei Kang: Writing – original draft, Writing – review \& editing; Zhimin Zhang: Writing – original draft, Writing – review \& editing.

\section*{Acknowledgment}
This research is supported by the National Natural Science Foundation of China (Grant Nos. 12494555, 12131005, 12201547), the Andrew Sisson Fund,  the Faculty Science Researcher Development Grant of the University of Melbourne, and the National Key R\&D Program of China (No. 2024YFA1016300).

\bibliographystyle{siamplain}
\bibliography{ref}

\begin{thebibliography}{10}

\bibitem{AKJKK2024}
{\sc A.~Abdulhaque, T.~Kvamsdal, K.~A. Johannessen, M.~Kumar, and A.~M.
  Kvarving}, {\em An efficient and easy-to-implement recovery-based a
  posteriori error estimator for isogeometric analysis of the {S}tokes
  equation}, Comput. Methods Appl. Mech. Engrg., 424 (2024), pp.~Paper No.
  116889, 40.

\bibitem{BS2001}
{\sc I.~Babu\v{s}ka and T.~Strouboulis}, {\em The finite element method and its
  reliability}, Numerical Mathematics and Scientific Computation, The Clarendon
  Press, Oxford University Press, New York, 2001.

\bibitem{BBCHS2006}
{\sc Y.~Bazilevs, L.~Beir\~ao~da Veiga, J.~A. Cottrell, T.~J.~R. Hughes, and
  G.~Sangalli}, {\em Isogeometric analysis: approximation, stability and error
  estimates for {$h$}-refined meshes}, Math. Models Methods Appl. Sci., 16
  (2006), pp.~1031--1090.

\bibitem{BGAR2022}
{\sc R.~Bharali, S.~Goswami, C.~Anitescu, and T.~Rabczuk}, {\em A robust
  monolithic solver for phase-field fracture integrated with fracture energy
  based arc-length method and under-relaxation}, Comput. Methods Appl. Mech.
  Engrg., 394 (2022), pp.~Paper No. 114927, 23.

\bibitem{BB1994}
{\sc T.~Blacker and T.~Belytschko}, {\em Superconvergent patch recovery with
  equilibrium and conjoint interpolant enhancements}, Internat. J. Numer.
  Methods Engrg., 37 (1994), pp.~517--536.

\bibitem{2006mixed}
{\sc D.~Boffi, F.~Brezzi, L.~F. Demkowicz, R.~G. Dur\'an, R.~S. Falk, and
  M.~Fortin}, {\em Mixed finite elements, compatibility conditions, and
  applications}, vol.~1939 of Lecture Notes in Mathematics, Springer-Verlag,
  Berlin; Fondazione C.I.M.E., Florence, 2008.
\newblock Lectures given at the C.I.M.E. Summer School held in Cetraro, June
  26--July 1, 2006.

\bibitem{BCS2010}
{\sc A.~Buffa, D.~Cho, and G.~Sangalli}, {\em Linear independence of the
  {T}-spline blending functions associated with some particular {T}-meshes},
  Comput. Methods Appl. Mech. Engrg., 199 (2010), pp.~1437--1445.

\bibitem{CGPS2017}
{\sc A.~Cangiani, E.~H. Georgoulis, T.~Pryer, and O.~J. Sutton}, {\em A
  posteriori error estimates for the virtual element method}, Numer. Math., 137
  (2017), pp.~857--893.

\bibitem{Ci1978}
{\sc P.~G. Ciarlet}, {\em The finite element method for elliptic problems},
  vol.~Vol. 4 of Studies in Mathematics and its Applications, North-Holland
  Publishing Co., Amsterdam-New York-Oxford, 1978.

\bibitem{IGAbook2009}
{\sc J.~A. Cottrell, T.~J.~R. Hughes, and Y.~Bazilevs}, {\em Isogeometric
  analysis}, John Wiley \& Sons, Ltd., Chichester, 2009.
\newblock Toward integration of CAD and FEA.

\bibitem{DCF2006}
{\sc J.~Deng, F.~Chen, and Y.~Feng}, {\em Dimensions of spline spaces over
  {$T$}-meshes}, J. Comput. Appl. Math., 194 (2006), pp.~267--283.

\bibitem{Deng2008}
{\sc J.~Deng, F.~Chen, X.~Li, C.~Hu, W.~Tong, Z.~Yang, and Y.~Feng}, {\em
  Polynomial splines over hierarchical {T}-meshes}, Graphical models, 70
  (2008), pp.~76--86.

\bibitem{GXZ2019}
{\sc H.~Guo, C.~Xie, and R.~Zhao}, {\em Superconvergent gradient recovery for
  virtual element methods}, Math. Models Methods Appl. Sci., 29 (2019),
  pp.~2007--2031.

\bibitem{GUO2025}
{\sc H.~Guo and Z.~Zhang}, {\em Chapter nine - recovery techniques for finite
  element methods}, in Error Control, Adaptive Discretizations, and
  Applications, Part 3, F.~Chouly, S.~P. Bordas, R.~Becker, and P.~Omnes, eds.,
  vol.~60 of Advances in Applied Mechanics, Elsevier, 2025, pp.~399--463.

\bibitem{GZZ2017}
{\sc H.~Guo, Z.~Zhang, and R.~Zhao}, {\em Hessian recovery for finite element
  methods}, Math. Comp., 86 (2017), pp.~1671--1692.

\bibitem{Hughes2005}
{\sc T.~J.~R. Hughes, J.~A. Cottrell, and Y.~Bazilevs}, {\em Isogeometric
  analysis: {CAD}, finite elements, {NURBS}, exact geometry and mesh
  refinement}, Comput. Methods Appl. Mech. Engrg., 194 (2005), pp.~4135--4195.

\bibitem{KXCD2015}
{\sc H.~Kang, J.~Xu, F.~Chen, and J.~Deng}, {\em A new basis for
  {PHT}-splines}, Graphical models, 82 (2015), pp.~149--159.

\bibitem{KKJ2017}
{\sc M.~Kumar, T.~Kvamsdal, and K.~A. Johannessen}, {\em Superconvergent patch
  recovery and a posteriori error estimation technique in adaptive isogeometric
  analysis}, Comput. Methods Appl. Mech. Engrg., 316 (2017), pp.~1086--1156.

\bibitem{survey}
{\sc X.~Li, F.~Chen, H.~Kang, and J.~Deng}, {\em A survey on the local
  refinable splines}, Sci. China Math., 59 (2016), pp.~617--644.

\bibitem{LS2014}
{\sc X.~Li and M.~A. Scott}, {\em Analysis-suitable {T}-splines:
  characterization, refineability, and approximation}, Math. Models Methods
  Appl. Sci., 24 (2014), pp.~1141--1164.

\bibitem{NZ2004}
{\sc A.~Naga and Z.~Zhang}, {\em A posteriori error estimates based on the
  polynomial preserving recovery}, SIAM J. Numer. Anal., 42 (2004),
  pp.~1780--1800.

\bibitem{NZ2005}
{\sc A.~Naga and Z.~Zhang}, {\em The polynomial-preserving recovery for higher
  order finite element methods in 2{D} and 3{D}}, Discrete Contin. Dyn. Syst.
  Ser. B, 5 (2005), pp.~769--798.

\bibitem{NK2011}
{\sc N.~Nguyen-Thanh, J.~Kiendl, H.~Nguyen-Xuan, R.~W\"uchner, K.~U.
  Bletzinger, and Y.~Bazilevs}, {\em Rotation free isogeometric thin shell
  analysis using {PHT}-splines}, Comput. Methods Appl. Mech. Engrg., 200
  (2011), pp.~3410--3424.

\bibitem{NN2011}
{\sc N.~Nguyen-Thanh, H.~Nguyen-Xuan, S.~P.~A. Bordas, and T.~Rabczuk}, {\em
  Isogeometric analysis using polynomial splines over hierarchical {T}-meshes
  for two-dimensional elastic solids}, Comput. Methods Appl. Mech. Engrg., 200
  (2011), pp.~1892--1908.

\bibitem{NS1974}
{\sc J.~A. Nitsche and A.~H. Schatz}, {\em Interior estimates for
  {R}itz-{G}alerkin methods}, Math. Comp., 28 (1974), pp.~937--958.

\bibitem{SW1977}
{\sc A.~H. Schatz and L.~B. Wahlbin}, {\em Interior maximum norm estimates for
  finite element methods}, Math. Comp., 31 (1977), pp.~414--442.

\bibitem{SW1995}
{\sc A.~H. Schatz and L.~B. Wahlbin}, {\em Interior maximum-norm estimates for
  finite element methods. {II}}, Math. Comp., 64 (1995), pp.~907--928.

\bibitem{SLSH2012}
{\sc M.~A. Scott, X.~Li, T.~W. Sederberg, and T.~J.~R. Hughes}, {\em Local
  refinement of analysis-suitable {T}-splines}, Comput. Methods Appl. Mech.
  Engrg., 213/216 (2012), pp.~206--222.

\bibitem{Sederberg2004}
{\sc T.~W. Sederberg, D.~L. Cardon, G.~T. Finnigan, N.~S. North, J.~Zheng, and
  T.~Lyche}, {\em T-spline simplification and local refinement}, ACM
  transactions on graphics (TOG), 23 (2004), pp.~276--283.

\bibitem{Sederberg2003}
{\sc T.~W. Sederberg, J.~Zheng, A.~Bakenov, and A.~Nasri}, {\em T-splines and
  {T-NURCCs}}, ACM transactions on graphics (TOG), 22 (2003), pp.~477--484.

\bibitem{Tian2011}
{\sc L.~Tian, F.~Chen, and Q.~Du}, {\em Adaptive finite element methods for
  elliptic equations over hierarchical {T}-meshes}, J. Comput. Appl. Math., 236
  (2011), pp.~878--891.

\bibitem{Wa1995}
{\sc L.~B. Wahlbin}, {\em Superconvergence in {G}alerkin finite element
  methods}, vol.~1605 of Lecture Notes in Mathematics, Springer-Verlag, Berlin,
  1995.

\bibitem{wang2010adaptive}
{\sc J.~Wang, Z.~Yang, L.~Jin, J.~Deng, and F.~Chen}, {\em Adaptive surface
  reconstruction based on implicit pht-splines}, in Proceedings of the 14th ACM
  Symposium on Solid and Physical Modeling, ACM New York, NY, USA, 2010,
  pp.~101--110.

\bibitem{XZ2003}
{\sc J.~Xu and L.~Zikatanov}, {\em Some observations on {B}abu\v ska and
  {B}rezzi theories}, Numer. Math., 94 (2003), pp.~195--202.

\bibitem{ZN2005}
{\sc Z.~Zhang and A.~Naga}, {\em A new finite element gradient recovery method:
  superconvergence property}, SIAM J. Sci. Comput., 26 (2005), pp.~1192--1213.

\bibitem{ZC2017}
{\sc Y.~Zhu and F.~Chen}, {\em Modified bases of {PHT}-splines}, Commun. Math.
  Stat., 5 (2017), pp.~381--397.

\bibitem{ZZ1992}
{\sc O.~C. Zienkiewicz and J.~Z. Zhu}, {\em The superconvergent patch recovery
  and a posteriori error estimates. {I}. {T}he recovery technique}, Internat.
  J. Numer. Methods Engrg., 33 (1992), pp.~1331--1364.

\bibitem{ZZ21992}
{\sc O.~C. Zienkiewicz and J.~Z. Zhu}, {\em The superconvergent patch recovery
  and a posteriori error estimates. {II}. {E}rror estimates and adaptivity},
  Internat. J. Numer. Methods Engrg., 33 (1992), pp.~1365--1382.

\end{thebibliography}

\end{document}